\documentclass{article}
\usepackage[utf8]{inputenc}
\usepackage[margin=1in]{geometry} 
\usepackage{graphicx,epstopdf}

\usepackage[utf8]{inputenc}
\usepackage[english]{babel}
\usepackage{amsmath,amssymb,epsfig,amsthm}
\usepackage{graphicx}
\usepackage{comment}
\usepackage{array}
\usepackage{algorithm}
\usepackage{enumerate}
\usepackage{xcolor}
\usepackage{algpseudocode}
\usepackage{xurl}
\usepackage{bbm}
\usepackage[round]{natbib}
\usepackage{graphicx}
\usepackage{subcaption}
\usepackage{algpseudocode}
\usepackage{tikz}
\usetikzlibrary{matrix,positioning,quotes}

\def\calZ{\mathcal{Z}}

\def\id{\operatorname{Id}}
\newtheorem{theorem}{Theorem}

\newtheorem{claim}{Claim}

\newtheorem{corollary}{Corollary}

\newtheorem{lemma}{Lemma}

\newtheorem{proposition}{Proposition}
\newtheorem{remark}{Remark}

\numberwithin{equation}{section}

\newcommand{\calB}{\ensuremath{\mathcal{B}}}

\newcommand{\calH}{\ensuremath{\mathcal{H}}}

\newcommand{\calP}{\ensuremath{\mathcal{P}}}

\newcommand{\calS}{\ensuremath{\mathcal{S}}}

\newcommand{\calN}{\ensuremath{\mathcal{N}}}

\newcommand{\calE}{\ensuremath{\mathcal{E}}}

\newcommand{\norm}[1]{\left\|{#1}\right\|}

\newcommand{\ceil}[1]{\lceil{#1}\rceil}

\newcommand{\expec}{\ensuremath{\mathbb{E}}}
\newcommand{\matR}{\ensuremath{\mathbb{R}}}

\newcommand{\argmax}[1]{\underset{#1}{\operatorname{argmax}}}
\newcommand{\argmin}[1]{\underset{#1}{\operatorname{argmin}}}

\newcommand{\prob}{\ensuremath{\mathbb{P}}}

\newcommand{\indic}{\ensuremath{\mathbf{1}}} 

\newcommand{\ppm}{\texttt{PPMGM}}

\newcommand{\rev}[1]{\textcolor{black}{#1}}
\newcommand{\erdos}{Erd\H{o}s}

\newcommand{\revb}[1]{\textcolor{black}{#1}}
\begin{document}

\title{Seeded graph matching for the correlated \rev{Gaussian} Wigner  model via the projected power method}

\author{Ernesto Araya\footnotemark[1]\\ \texttt{ernesto-javier.araya-valdivia@inria.fr} \and Guillaume Braun \footnotemark[2]\\  \texttt{guillaume.braun@riken.jp} \and Hemant Tyagi \footnotemark[1]\\  \texttt{hemant.tyagi@inria.fr}}

\renewcommand{\thefootnote}{\fnsymbol{footnote}}
\footnotetext[1]{Inria, Univ. Lille, CNRS, UMR 8524 - Laboratoire Paul Painlev\'{e}, F-59000 }
\footnotetext[2]{ Riken - AIP, Tokyo}

\renewcommand{\thefootnote}{\arabic{footnote}}


\maketitle

\begin{abstract}
In the \emph{graph matching} problem we observe two graphs $G,H$ and the goal is to find an assignment (or matching) between their vertices such that some measure of edge agreement is maximized. We assume in this work that the observed pair $G,H$ has been drawn from the Correlated Gaussian Wigner (CGW) model -- a popular model for correlated weighted graphs -- where the entries of the adjacency matrices of $G$ and $H$ are independent Gaussians and each edge of $G$ is correlated with one edge of $H$ (determined by the unknown matching) with the edge correlation  described by a parameter $\sigma\in [0,1)$. In this paper, we analyse the performance of the \emph{projected power method} (PPM) as a \emph{seeded} graph matching algorithm where we are given an initial partially correct matching (called the seed) as side information. We prove that if the seed is close enough to the ground-truth matching, then with high probability, PPM iteratively improves the seed and recovers the ground-truth matching (either partially or exactly) in $\mathcal{O}(\log n)$ iterations. Our results prove that PPM works even in regimes of constant $\sigma$, thus extending the analysis in \citep{MaoRud} for the sparse Correlated \erdos-Renyi (CER) model to the (dense) CGW model. As a byproduct of our analysis, we see that the PPM framework generalizes some of the state-of-art algorithms for seeded graph matching. We support and complement our theoretical findings with numerical experiments on synthetic data. 
\end{abstract}

\section{Introduction}\label{sec:intro}
In the \emph{graph matching problem} we are given as input two graphs $G$ and $H$ with an equal number of vertices, and the objective is to find a bijective function, or \emph{matching}, between the vertices of $G$ and $H$ such that the alignment between the edges of $G$ and $H$ is maximized. This problem appears in many applications such as computer vision \citep{sunfei}, network de-anonymization \citep{nay}, pattern recognition \citep{conte, streib}, protein-protein interactions and computational biology \citep{zasla,singh}. In computer vision, for example, it is used as a method of comparing two objects (or images) encoded as graph structures or to identify the correspondence between the points of two discretized images of the same object at different times. In network de-anonymization, the goal is to learn information about an anonymized (unlabeled) graph using a related labeled graph as a reference, exploiting their structural similarities. In \citep{naya2} for example, the authors show that it was possible to effectively de-anonymize the Netflix database using the IMDb (Internet Movie Database) as the ``reference'' network. 

The graph matching problem is well defined for any pair of graphs (weighted or unweighted) and it can be framed as an instance of the NP-hard \emph{quadratic assignment problem} (QAP) \citep{QAPhard2014}. It also contains the ubiquitous \emph{graph isomorphism} (with unknown complexity) as a special case. However, in the average case situation, many polynomial time algorithms have recently been shown to recover, either perfectly or partially, the ground-truth vertex matching with high probability. It is thus customary to assume that the observed graphs $G,H$ are generated by a model for correlated random graphs, where the problem can be efficiently solved. The two most popular models are the correlated Correlated \erdos-R\'enyi \revb{(CER)} model \citep{PedGloss}, where two graphs are independently sampled from an \erdos-R\'enyi mother graph, and the Correlated Gaussian Wigner \revb{(CGW)} model \citep{deg_prof,Grampa}, which considers that $G,H$ are complete weighted graphs with independent Gaussian entries on each edge; see Section \ref{sec:RGmodels} for a precise description. Recently, other models of correlation have been proposed for random graphs with a latent geometric structure \citep{geo_1,geo_2}, community structure \citep{AniRac} and with power law degree profile \citep{Powerlaw}.  

In this work, we will focus on the \emph{seeded} version of the problem, where side information about the matching is provided (together with the two graphs $G$ and $H$) by a partially correct bijective map from the vertices of $G$ to the vertices of $H$ referred to as the seed.
The quality of the seed can be measured by its overlap with the ground-truth matching. This  definition of a seed is more general than what is often considered in the literature \citep{MosselXu}, including the notion of a partially correct (or noisy) seed \citep{LubSri,YuXuLin}. The seeded version of the problem is motivated by the fact that in many applications, a set of correctly matched vertices is usually available -- either as prior information, or it can be constructed by hand (or via an algorithm). From a computational point of view, seeded algorithms are also more efficient than seedless algorithms (see the related work section).
Several algorithms, based on different techniques, have been proposed for seeded graph matching. In \citep{PedGloss,YarGross}, the authors use a percolation-based method to ``grow'' the seed to recover (at least partially) the ground-truth matching. Other algorithms \citep{LubSri,YuXuLin} construct a similarity matrix between the vertices of both graphs and then solve the maximum linear assignment problem (either optimally or by a greedy approach) using the similarity matrix as the cost matrix. The latter strategy has also been successfully applied in the case described below when no side information is provided. \cite{MaoRud} also analyzed an iterative refinement algorithm to achieve exact recovery under the CER model with a constant level of correlation. However, all of the previously mentioned work focuses on binary (and often sparse) graphs while in a lot of applications, for example, protein-to-protein interaction networks or image matching, the graphs of interest are weighted and sometimes dense.


\paragraph{Contributions.} The main contributions of this paper are summarized below. To our knowledge, these are the first theoretical results for seeded \emph{weighted} graph matching. 
\begin{itemize}
    
    \item We analyze a variant of the \emph{projected power method} (PPM)  for the seeded graph matching problem in the context of the CGW model.
    We provide (see Theorems \ref{prop:one_it_conv}, \ref{prop:partial_rec}) exact and partial recovery guarantees under the CGW model when the PPM is initialized with a given data-independent seed, and only one iteration of the PPM algorithm is performed. For this result to hold, it suffices that the overlap of the seed with the ground-truth permutation is $\Omega(\sqrt{n \log n})$. 
    
    \item We prove (see Theorem \ref{thm:unif_rec_ppm}) that when multiple iterations are allowed, then PPM converges to the ground-truth matching in $\mathcal{O}(\log n)$ iterations  provided that it is initialized with a seed with overlap $\Omega\big((1-\kappa)n\big)$, for a constant $\kappa$ small enough, even if the initialization is data-dependent (i.e. a function of the graphs to be matched) or adversarial. This extends the results in \citep{MaoRud} from the sparse \erdos-R\'enyi setting, to the dense Gaussian Wigner case. 
    
    \item We complement our theoretical results with experiments on synthetic data, showing that PPM can help to significantly improve the accuracy of the matching (for the Correlated Gaussian Wigner model) compared to that obtained by a standalone application of existing seedless methods.
\end{itemize}


\subsection{Notation}\label{sec:notation}
We denote $\calP_n$ to be the set of permutation matrices of size $n\times n$ and $\calS_n$ the set of permutation maps on the set $[n]=\{1,\cdots,n\}$. To each element $X \in \calP_n$ (we reserve capital letters for its matrix form), there corresponds one and only one element $x\in \calS_n$ (we use lowercase letters when referring to functions). We denote $\operatorname{Id}$ (resp. $\operatorname{id}$) the identity matrix (resp. identity permutation), where the size will be clear from the context. For $X\in \calP_n$($x\in\calS_n$), we define $S_X=\{i\in[n]:X_{ii}=1\}$ to be the set of fixed points of $X$, and $s_x=|S_X|/n$ its fraction of fixed points. The symbols $\langle \cdot,\cdot\rangle_F$, and $\|\cdot\|_F$ denote the Frobenius inner product and its induced matrix norm, respectively. For any matrix $X\in\mathbb{R}^{n\times n}$, let $[X]\in \mathbb{R}^{n^2}$ denote its vectorization obtained by stacking its columns one on top of another. For two random variables $X,Y$ we write $X\stackrel{d}{=}Y$ when they are equal in law. For a matrix $A\in \mathbb{R}^{n\times n}$, $A_{i:}$ (resp. $A_{:i}$) will denote its $i$-th row (resp. column). 

\subsection{Mathematical description}
Let $A,B$ be the adjacency matrices of the graphs $G,H$ each with $n$ vertices. In the graph matching problem, the goal is to find the solution of the following optimization problem
\begin{equation}\label{form:1}
    \max_{x\in \mathcal{S}_n}\sum_{i,j}A_{ij}B_{x(i)x(j)} \enskip\tag{P1}
\end{equation}
 which is equivalent to solving
\begin{equation}\label{form:1'}
    \max_{X\in \mathcal{P}_n}\langle A,X BX^\top\rangle_F. \tag{P1'}
\end{equation}
Observe that \eqref{form:1} is a well-defined problem -- not only for adjacency matrices -- but for any pair of matrices of the same size. In particular, it is well-defined when $A,B$ are adjacency matrices of weighted graphs, which is the main setting of this paper. Moreover, this is an instance of the well-known \emph{quadratic assignment problem}, which is a combinatorial optimization problem known to be NP-hard in the worst case \citep{QAP}. 
Another equivalent formulation of \eqref{form:1} is given by the following ``lifted'' (or vector) version of the problem

\begin{equation}
    \label{form:1''}
    \max_{[X]\in [\mathcal{P}_n]}[X]^\top(B\otimes A)[X]\tag{P1''}
\end{equation}
where $[\mathcal{P}_n]$ is the set of permutation matrices in vector form. This form has been already considered in the literature, notably in the family of spectral methods \citep{Villar,spec_align}.


\subsection{Statistical models for correlated random graphs}\label{sec:RGmodels}
Most of the theoretical statistical analysis for the graph matching problem has been performed so far under two random graph models: the \emph{Correlated \erdos-R\'enyi} and the \emph{Correlated Gaussian Wigner models}. In these models the dependence between the two graphs $A$ and $B$ is explicitly described by the inclusion of a ``noise'' parameter which captures the degree of correlation between $A$ and $B$. 

\paragraph{Correlated Gaussian Wigner (CGW) model $W(n,\sigma,x^*)$.} The problem \eqref{form:1} is well-defined for matrices that are not necessarily $0/1$ graph adjacencies, so a natural extension is to consider two complete weighted graphs. The following Gaussian model has been proposed in \citep{deg_prof}  \[A_{ij}\sim\begin{cases}\calN(0,\frac1n)\text{ if }i<j, \\ \calN(0,\frac2n)\text{ if } i= j,
\end{cases}\]
$A_{ij}=A_{ji}$ for all $i,j\in [n]$, and $B_{x^*(i)x^*(j)}=\sqrt{1-\sigma^2}A_{ij}+\sigma Z_{ij}$, where  $Z\stackrel{d}{=}A$. Both $A$ and $B$ are distributed as the GOE (Gaussian orthogonal ensemble). Here the parameter $\sigma>0$ should be interpreted as the noise parameter and in that sense, $B$ can be regarded as a ``noisy perturbation'' of $A$. Moreover, $x^* \in \calS_n$ is the ground-truth (or latent) permutation that we seek to recover. It is not difficult to verify that the problem \eqref{form:1} is in fact the maximum likelihood estimator (MLE) of $x^*$ under the CGW model.

\paragraph{Correlated \erdos-R\'enyi (CER) model $G(n,q,s,x^*)$}. For $q,s\in[0,1]$, the correlated \erdos-R\'enyi model with latent permutation $x^*\in \calS_n$ can be described in two steps.
\begin{enumerate}
    \item $A$ is generated according to the \erdos-R\'enyi model $G(n,q)$, i.e. for all $i<j$, $A_{ij}$ is sampled from independent Bernoulli's r.v. with parameter $q$, $A_{ji}=A_{ij}$ and $A_{ii}=0$.
 
    \item Conditionally on $A$, the entries of $B$ are i.i.d according to the law %
    \begin{equation}\label{eq: ER_def}
        B_{x^*(i),x^*(j)}\sim\begin{cases}
        Bern(s)\quad \text{if}\quad A_{ij}=1,\\
        Bern\big(\frac{q}{1-q}(1-s)\big)\quad \text{if } A_{ij}=0.
        \end{cases}
    \end{equation}
\end{enumerate}
There is another equivalent description of this model in the literature, where to obtain CER graphs, we first sample an \erdos-R\'enyi ``mother'' graph and then define $A,B$ as independent subsamples with certain density parameter. We refer to \citep{PedGloss} for details.


\subsection{Related work} \label{sec:rel_work}
%


\paragraph{Information-theoretic limits of graph matching.} The necessary and sufficient conditions for correctly estimating the matching between two graphs when they are generated from the CGW or the CER model have been investigated in \citep{CullKi,HallMass,recons_thr}. In particular,  for the CGW model, it has been shown in \citep[Thm.1]{recons_thr} that the ground truth permutation $x^*$ can be exactly recovered w.h.p. only when $\sigma^2\leq 1-\frac{(4+\epsilon)\log n}{n}$. When $\sigma^2\geq 1-\frac{(4-\epsilon)\log n}{n}$ no algorithm can even partially recover $x^*$. However, it is not known if there is a polynomial time algorithm that can reach this threshold. 

\paragraph{Efficient algorithms}
\begin{itemize}
    \item \textbf{Seedless algorithms.} Several polynomial time algorithms have been proposed relying on spectral methods \citep{Spectral_weighted_Ume,Grampa,ganMass,spec_align,Balanced_GM}, degree profiles \citep{deg_prof,dai_cullina}, other vertex signatures \citep{MaoRud}, random walk based approaches \citep{isorank_1,isorank_2,Gori04graphmatching}, convex and concave relaxations \citep{afla,Lyzin,bach}, and other non-convex methods \citep{YuYan,XuLuo,Villar}. Most of the previous algorithms have theoretical guarantees only in the low noise regime. For instance, the \texttt{Grampa} algorithm proposed in \citep{Grampa} provably exactly recovers the ground truth permutation for the CGW model when $\sigma=\mathcal{O}(\frac1{\log n})$, and in \citep{deg_prof} it is required for the CER (resp. CGW) model that the two graphs differ by at most $1-s=\mathcal{O}(\frac1{\log^2 n})$ fraction of edges (resp. $\sigma=O(\frac{1}{\log n})$).  
    There are only two exceptions for the CER model where the noise level is constant: the work of \citep{ganMass2} and \citep{MaoRud}. But these algorithms exploit the sparsity of the graph in a fundamental manner and cannot be extended to dense graphs.
    
    \item \textbf{Seeded algorithms.} In the seeded case, different kinds of consistency guarantees have been proposed: consistency after one refinement step \citep{YuXuLin,LubSri}, consistency after several refinement steps uniformly over the seed \citep{MaoRud}. For the dense CER ($p$ of constant order), one needs to have an initial seed with $\Omega (\sqrt{n\log n})$ overlap in order to have consistency after one step for a given seed \citep{YuXuLin}. But if we want to have a uniform result, one needs to have a seed that overlaps the ground-truth permutation in $O(n)$ points \citep{MaoRud}. Our results for the CGW model are similar in that respect. Besides, contrary to seedless algorithms, our algorithm works even if the the noise level $\sigma$ is constant.  
\end{itemize}

\paragraph{Projected power method (PPM).} PPM, which is also often referred to as a  \emph{generalized power method} (GPM) in the literature is a family of iterative algorithms for solving constrained optimization problems. It has been used with success for various tasks including clustering SBM \citep{Wang2021OptimalNE}, group synchronization \citep{boumal2016,GaoZhang}, joint alignment from pairwise difference \citep{chen2016_alignment}, low rank-matrix recovery \citep{chi2019} and the generalized orthogonal Procrustes problem \citep{Ling}. It is a useful iterative strategy for solving non-convex optimization problems, and usually  requires a good enough initial estimate. In general, we start with an initial candidate satisfying a set of constraints and at each iteration we perform 
\begin{enumerate}
\item a \emph{power step}, which typically consists in multiplying our initial candidate with one or more data dependent matrices, and  

\item a \emph{projection step} where the result of the power step is projected onto the set of constraints of the optimization problem. 
\end{enumerate}
These two operations are iteratively repeated and often convergence to the ``ground-truth signal'' can be ensured in $\mathcal{O}(\log n)$ iterations, provided that a reasonably good initialization is provided. 

The projected power method (PPM) has also been used to solve the graph matching problem, and its variants, by several authors. In some works, it has been explicitly mentioned \citep{Villar,Hippi}, while in others \citep{MaoRud, YuXuLin,LubSri} very similar algorithms have been proposed without acknowledging the relation with PPM (which we explain in more detail below). All the works that study PPM explicitly do not report statistical guarantees and, to the best of our knowledge, theoretical guarantees have been obtained only in the case of sparse \erdos-R\'enyi graphs, such as in \citep[Thm.B]{MaoRud} in the case of multiple iterations, and \citep{YuXuLin,LubSri} in the case of one iteration. Interestingly, the connection with the PPM is not explicitly stated in any of these works.



\section{Algorithm}\label{sec:alg_res}

\subsection{Projected power method for Graph matching}\label{sec:ppmgm}



We start by defining the projection operator onto $\calP_n$ for a matrix $C\in \matR^{n\times n}$. We will use the greedy maximum weight matching (GMWM) algorithm introduced in \citep{LubSri}, for the problem of graph matching with partially correct seeds, and subsequently used in \citep{YuXuLin}. The steps are outlined in Algorithm \ref{alg:gmwm}. 
\begin{algorithm}
\caption{\texttt{GMWM} (Greedy maximum weight matching)}\label{alg:gmwm}
\begin{algorithmic}[1]
  \Require{A cost matrix $C\in\mathbb{R}^{n\times n}$.}
  \Ensure{A permutation matrix $X$.}
  \State Select $(i_1,j_1)$ such that $C_{i_1,j_1}$ is the largest entry in $C$ (break ties arbitrarily). Define $C^{(1)}\in\mathbb{R}^{n\times n}$: $C^{(1)}_{ij}=C_{ij}\mathbbm{1}_{i\neq i_1,j\neq j_1}-\infty\cdot\mathbbm{1}_{i=i_1\text{or } j= j_1}$.
  
  \For{$k=2$ to $n$}
        \State  Select $(i_k,j_k)$ such that $C^{(k-1)}_{i_k,j_k}$ is the largest entry in $C^{(k-1)}$.
        
        \State Define $C^{(k)}\in\mathbb{R}^{n\times n}$: $C^{(k)}_{ij}=C^{(k-1)}_{ij}\mathbbm{1}_{i\neq i_k,j\neq j_k}-\infty\cdot\mathbbm{1}_{i=i_k\text{or } j= j_k}$.
        
      \EndFor 
      \State Define $X\in \{0,1\}^{n\times n}$: $X_{ij}=\sum^n_{k=1}\mathbbm{1}_{i=i_k,j=j_k}$.
      \State\Return{$X$}
  \end{algorithmic}
\end{algorithm}
Notice that the original version of GMWM works by erasing the row and column of the largest entry of the matrix $C^{(k)}$ at each step $k$. We change this to assign $-\infty$ to each element of the row and column of the largest entry (which is equivalent), mainly to maintain the original indexing. 
The output of Algorithm \ref{alg:gmwm} is clearly a permutation matrix, hence we define 
\begin{equation}\label{eq:projection}
    \tau (C):=\text{Output of GMWM with input } C
\end{equation}
which can be considered a projection since $\tau(\tau(C))=\tau(C)$ for all $C\in\mathbb{R}^{n\times n}$.
Notice that, in general, the output of GMWM will be different from solving the linear assignment problem  
\begin{align*}
\tilde{\tau}(C): =\argmin{}{\{\|C-X\|_F\enskip |\enskip X\in\calP_n\}}  
=\argmax{\Pi\in\calP_n}{\langle \Pi,C \rangle_F}\nonumber
\end{align*}
which provides an orthogonal projection, while $\tau$ corresponds to an oblique projection in general.

\begin{algorithm}
\caption{\texttt{PPMGM} (PPM for graph matching)}\label{alg:ppmgm}
\begin{algorithmic}[1]
  \Require{Matrices $A,B$, an initial point $X^{(0)}$ and $N$ the maximum number of iterations.}
  \Ensure{A permutation matrix $X$.}
 
  \For{$k=0$ to $N-1$}
        \State  $X^{(k+1)} \gets \tau(AX^{(k)}B)$.
  \EndFor 
  \State\Return{$X=X^{(N)}$}
  \end{algorithmic}
\end{algorithm}
The PPM is outlined in Algorithm \ref{alg:ppmgm}. Given the estimate of the permutation $X^{(k)}$ from step $k$, the power step corresponds to the operation $AX^{(k)}B$ while the projection step is given by the application of the projection $\tau$ on $AX^{(k)}B$. 
The similarity  matrix $C^{(k+1)}:=AX^{(k)}B$ is the matrix form of the left multiplication of $[X^{(k)}]$ by the matrix $B\otimes A$. Indeed, given that $A$ and $B$ are symmetric matrices, we have $[AX^{(k)}B]=(B\otimes A)[X^{(k)}]$, by \citep[eqs. 6 and 10]{Schacke}. All previous works related to the PPM for graph matching \citep{Villar}, and its variants \citep{Hippi}, use $(B\otimes A)[X^{(k)}]$ in the power step which is highly inconvenient in practice. This is because the matrix $B\otimes A$ is  expensive to store and do computations with if we follow the naive approach, consisting of computing and storing $B\otimes A$, and using it in the power step (we expand on this in Remark \ref{rem:complexity} below). Also, a power step of the form $AX^{(k)}B$ connects the PPM with the seeded graph matching methods proposed for the CER model  \citep{LubSri,YuXuLin,MaoRud}  where related similarity matrices are used, thus providing a more general framework. Indeed, in those works, the justification for the use of the matrix $AX^{(k)}B$ is related to the notion of witnesses (common neighbors between two vertices), which is an information that can be read in the entries of $AX^{(k)}B$. In our case, the similarity matrix $AX^{(k)}B$ appears naturally as the gradient of the objective function of \eqref{form:1'} and does not require the notion of witnesses (which does not extend automatically to the case of weighted matrices with possibly negative weights). In addition, the set of elements correctly matched by the initial permutation $x^{(0)}\in\calS_n$ will be defined here as the seed of the problem, i.e., we take the set of seeds $S:=\{(i,i'): x^{(0)}(i)=i'\}$. Thus, the number of correct seeds will be the number of elements $i\in [n]$ such that $x^{(0)}(i)=x^*(i)$. Observe that the definition of the seed as a permutation contains less information than the definition of a seed as a set $S$ of bijectively (and correctly) pre-matched vertices, because $S$ can be augmented (arbitrarily) to a permutation and, in addition, by knowing $S$ we have information on which vertices are correctly matched. We mention this as there is a commonly used definition of seed in the literature (see for example \citep{YarGross}) which considers that the set $S$ contains only correctly matched vertices, thus giving more information than what is necessary for our algorithm to work\footnote{In other words, our algorithm does not need to know which are the correct seeds, but only that there is a sufficiently large number of them in the initial permutation.}. The notion of permutation as a partially correct seed that we are using here has been used, for example, in \citep{YuXuLin}.

\paragraph{Initialization.} We prove in Section \ref{sec:conv_analysis} that Algorithm \ref{alg:ppmgm} recovers the ground truth permutation $x^*$ provided that the initialization $x^{(0)}$ is sufficiently close to $x^*$. The initialization assumption will be written in the form
\begin{equation}\label{assumption:init}
    \|X^{(0)}-X^*\|_F\leq \theta \sqrt n
\end{equation}
for some $\theta\in[0,\sqrt 2)$. Here, the value of $\theta$ measures how good $X^{(0)}$ is as a seed. Indeed, \eqref{assumption:init} can be equivalently stated as: the number of correct seeds is larger than $\frac n{\rev{2}}(2-\theta^2)$. The question of finding a good initialization method can be seen as a seedless graph matching problem, where only partial recovery guarantees are necessary. In practice, we can use existing seedless algorithms such as those in \citep{Spectral_weighted_Ume,Grampa,spec_align} to initialize Algorithm \ref{alg:ppmgm}. We compare different initialization methods numerically, in Section \ref{sec:experiments}. 

\begin{remark}[PPM as a gradient method]
The projected power method can be seen as a projected gradient ascent method for solving the MLE formulation in \eqref{form:1'}. 
From the formulation \eqref{form:1''} it is clear that the gradient of the likelihood evaluated on $X\in \calP_n$ is $2(B\otimes A)[X]$ or, equivalently, $2AXB$ in matrix form. This interpretation of PPM has been acknowledged in the context of other statistical problems \citep{jour,chen2016_alignment}. 
\end{remark}
\begin{remark}[Optimality]
Algorithms based on PPM or GPM have been shown to attain optimal, or near-optimal, statistical guarantees for several problems in statistics, including community detection \citep{Wang2021OptimalNE,WangManchoso}, group syncronization \citep{boumal2016,Gao2019IterativeAF} and generalized orthogonal procrustes problem \citep{Ling}. 
\end{remark}
\begin{remark}[Complexity]\label{rem:complexity}
The computational time complexity of Algorithm \ref{alg:ppmgm} is $\mathcal{O}(n^\omega\log{n}+n^2\log^2{n})$, where $\mathcal{O}(n^\omega)$ is the matrix multiplication complexity and $\mathcal{O}(n^2\log{n})$ is the complexity of Algorithm \ref{alg:gmwm} \citep{YuXuLin}. In \citep{Le_Gall}, the authors establish the bound $\omega\leq 2.373$. Notice that, in comparison to our algorithm, any algorithm using $B\otimes A$ in a naive way that involves first computing and storing $B\otimes A$ and then doing multiplications with it, has a time complexity at least $n^4$ (the cost of computing $B\otimes A$ when $A,B$ are dense matrices) 

\end{remark}


%
\section{Main results}\label{sec:conv_analysis}
Our goal in this section is to prove recovery guarantees for Algorithm \ref{alg:ppmgm} when the input matrices $A,B$ are realizations of the correlated Wigner model, described earlier in Section \ref{sec:RGmodels}. 
In what follows, we will assume without loss of generality that $X^*=\id$. Indeed, this can be seen by noting that for any permutation matrix $X$, and $C \in \mathbb{R}^{n \times n}$, it holds that $\tau(C X) = \tau(C) X$. This means that if we permute the rows and columns of $B$ by $X$ (by replacing $B$ with $X B X^\top$), and replace $X^{(0)}$ by $X^{(0)} X^\top$, then the output of Algorithm \ref{alg:ppmgm} will be given by the iterates $(X^{(k)} X^\top)_{k=1}$.

\subsection{Exact recovery in one iteration}\label{sec:mainstep1}
For any given seed $x^{(0)}$ that is close enough to $x^*$, the main result of this section states that $x^*$ is recovered exactly in one iteration of Algorithm \ref{alg:ppmgm} with high probability. Let us first introduce the following definition: we say that a matrix $M$ is diagonally dominant\footnote{This is weaker than the usual notion of diagonal dominance, where for all $i\in [n]$,  $|M_{ii}|\geq \sum_{j\neq i}|M_{ij}|$.} if for all $i,j$ with $i\neq j$ we have $M_{ii}>M_{ij}$. This notion will be used in conjunction with the following lemma, its proof is in Appendix \ref{app:proofs_lem_diagdom}.
\begin{lemma}\label{lem:diagdom_LAP}
 If a matrix $C$ satisfies the diagonal dominance property, then the greedy algorithm \texttt{GMWM} with input $C$ will return the identical permutation. Consequently, if $A,B\sim W(n,\sigma,\operatorname{id})$, then for $C=AXB$ and $\Pi=\tau(C)$, we have \begin{equation}\label{eq:probneqId}
    \prob(\Pi\neq \id)\leq \prob(C \textup{ is not diag. dominant})
\end{equation}  
\end{lemma}

The next theorem allow us to control the probability that $C$ is not diagonally dominant and, in turn, proves that Algorithm \ref{alg:ppmgm} recovers the ground truth permutation with high probability. The proof of Theorem \ref{prop:one_it_conv} is outlined in in Section \ref{sec:thm_one_it}.
\begin{theorem}\label{prop:one_it_conv}
Let $A,B\sim W(n,\sigma,\operatorname{id})$ and 
 $X\in\calP_n$ with $\|X-\operatorname{Id}\|_F\leq \theta\sqrt{n}$, with $0\leq \theta \leq\sqrt{2(1-\frac{10}n)}$ and $n\geq 10$. Then the following holds.
 \begin{enumerate}[(i)]
     \item For $C=AXB$ we have \[\prob(C \textup{ is not diag. dominant })\leq 5n^2e^{-c(\sigma)\big(1-\frac{\theta^2}{2}\big)^2n}\]
     where $c(\sigma)=\frac1{384}\left(\frac{1-\sigma^2}{1+2\sigma\sqrt{1-\sigma^2}}\right)$.
     
     \item Denote $\Pi$ as the output of Algorithm \ref{alg:ppmgm} with input $(A,B,X^{(0)}=X,N=1)$, then 
 \[\prob(\Pi=\operatorname{Id})\geq 1-5n^2e^{-c(\sigma)\big(1-\frac{\theta^2}2\big)^2n}.\]
 In particular, if $\|X-\operatorname{Id}\|^2_F\leq 2\Big(n-\sqrt{\frac{1}{c(\sigma)}n\log{(5n^3)}}\Big)$ then 
 \[\prob(\Pi=\operatorname{Id})\geq 1-n^{-1}.\]
 \end{enumerate}
 
\end{theorem}

\begin{remark}
The assumption $\|X-\operatorname{Id}\|^2_F\leq 2(n-\sqrt{\frac{1}{c(\sigma)}n\log{(5n^3)}})$ can be restated as $|S_X|\geq \sqrt{\frac{1}{c(\sigma)}n\log{5n^3}}$, where $S_X$ is the set of fixed points of $X$. That is, for this assumption to hold, we need that $X$ has a number of fixed points of order $\Omega_\sigma(\sqrt{n\log n})$.
Also note that $c(\sigma)$ is decreasing with $\sigma$, which is consistent with the intuition that larger levels of noise make it more difficult to recover the ground truth permutation. We include a plot of $c(\sigma)$ (rescaled) in Figure \ref{fig:c_2_sig}.
\begin{figure}[!ht]
    \centering
    \includegraphics[scale=0.29]{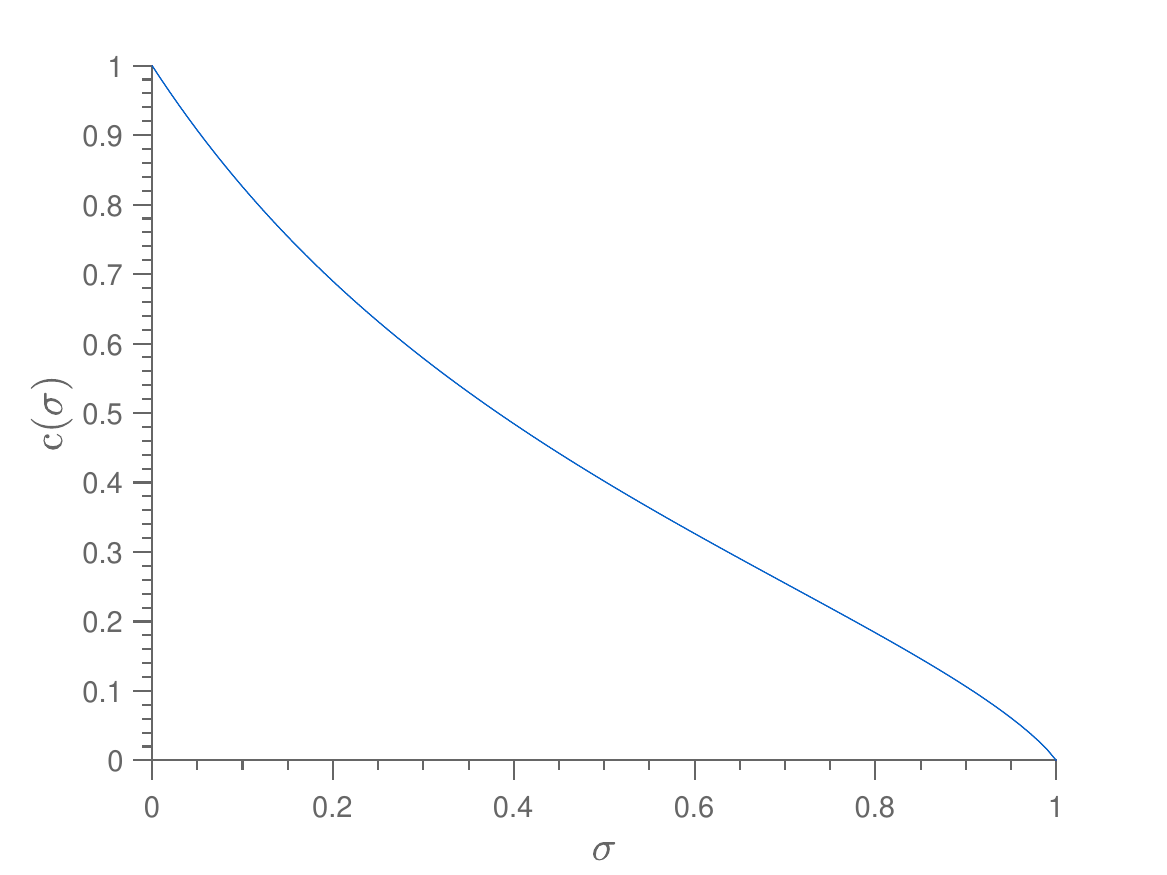}
    \caption{The constant $c(\sigma)$ (re-scaled multiplying by $384$) appearing in Theorem \ref{prop:one_it_conv}.}
    \label{fig:c_2_sig}
\end{figure}
\end{remark}

\paragraph{Discussion.} Given an initial seed $X^{(0)}\in \calP_n$, the case $N=1$ in Algorithm \ref{alg:ppmgm} can be alternatively interpreted as the following two step process: first, compute a similarity matrix $AX^{(0)}B$ and then round the similarity matrix to an actual permutation matrix. This strategy has been frequently applied in graph matching algorithms in both the seeded and seedless case \citep{Spectral_weighted_Ume,Grampa,LubSri,YuXuLin}. In terms of the  quality of the seed, Theorem \ref{prop:one_it_conv} gives the same guarantees obtained by \citep[Thm.1]{YuXuLin} which requires $\Omega(\sqrt{n\log n})$ vertices in the seed to be correctly matched. However the results of \citep{YuXuLin} are specifically for the correlated Erdös-Renyi model.

\subsection{Partial recovery in one iteration} \label{subsec:partial_one_step}
%
%
In the partial recovery setting, we are interested in the fraction of nodes that are correctly matched. To this end let us define the following measure of performance
\begin{equation}\label{eq:overlap_def}
    \operatorname{overlap}(\nu,\nu'):=\frac{1}{n}|\{i\in[n]:\nu(i)=\nu'(i)\}|
\end{equation}
for any pair $\nu,\nu'\in\calS_n$.
Recall that we assume that the ground truth permutation is $x^*=\operatorname{id}$ and $\pi$ is the output of Algorithm \ref{alg:ppmgm} with input $(A,B,X^{(0)}=X,N=1)$ where $\Pi=\operatorname{GMWM} (AXB)$. Observe that $\operatorname{overlap}(\pi,x^*=\operatorname{id})=s_\pi$ is the fraction of fixed points of the permutation $\pi$. It will be useful to consider the following definition. We say that $C_{ij}$ is \emph{row-column dominant} if $C_{ij}> C_{i'j}$ for all $i'\neq i$ and $C_{ij}>C_{ij'}$, for all $j'\neq j$. The following lemma  relates the overlap of the output of $\texttt{GMWM}$ with the property that a subset of the entries of $C$ is row-column dominant, its proof is outlined in Appendix \ref{app:proofs_lem_diagdom}.  
\begin{lemma}\label{lem:overlap_event}
Let $C$ be a $n\times n$ matrix with the property that there exists a set $\{i_1,\cdots,i_r\}$, with $1\leq r \leq n$ such that $C_{i_k,i_k}$ is row-column dominant for $k\in[r]$. Let $\pi\in\calS_n$ be permutation corresponding to $\operatorname{GMWM}(C)\in\calP_n$. Then it holds that $\pi(i_k)=i_k$ for $k\in[r]$ and, in consequence, the following event inclusion holds
\begin{equation}\label{eq:overlap_event}
    \{\operatorname{overlap}(\pi,\operatorname{id})< r/n\}\subset\bigcap_{\substack{I_r\subset [n]\\|I_r|=r}}\bigcup_{i\in I_r}\{C_{ii} \textup{ is not row-column dominant } \}.
\end{equation}
\end{lemma}

Equipped with this lemma, we can prove the following generalization of Theorem \ref{prop:one_it_conv}, its proof is detailed in Section \ref{subsec:proof_thm_partial_rec}.
\begin{theorem}\label{prop:partial_rec}
Let $A,B\sim W(n,\sigma,\operatorname{id})$ and $X\in \calP_n$ with $\|X-\operatorname{Id}\|_F\leq \theta\sqrt{n}$, where $0\leq \theta \leq\sqrt{2(1-\frac{10}n)}$ and $n\geq 10$. Let $\pi\in \calS_n$ be the output of Algorithm \ref{alg:ppmgm} with input $(A,B,X^{(0)}=X,N=1)$. Then, for $r\in[n]$
\begin{equation*}
  \prob( \operatorname{overlap}(\pi,\operatorname{id})> r/n)\geq 1-16rne^{-c(\sigma)\big(1-\frac{\theta^2}2\big)^2n}.
\end{equation*}
In particular, if $x\in\calS_n$ is the map corresponding to $X$ and  $|S_X|\geq \sqrt{\frac1{c(\sigma)}n\log{(16rn^2)}}$, then   
\begin{equation*}
  \prob( \operatorname{overlap}(\pi,\operatorname{id})> r/n)\geq 1-n^{-1}.
\end{equation*}
\end{theorem}



\subsection{Exact recovery after multiple iterations, uniformly in the seed}
The results in Sections \ref{sec:mainstep1} and \ref{subsec:partial_one_step} hold for any given seed $X^{(0)}$, and it is crucial that the seed does not depend on the graphs $A, B$. 
In this section, we provide uniform convergence guarantees for \ppm \ which hold uniformly over all choices of the seed in a neighborhood around $x^*$. %
\begin{theorem} \label{thm:unif_rec_ppm}
Let $\sigma \in [0,1)$, $A,B\sim W(n,\sigma,\operatorname{id})$ \revb{and define $\kappa:=\left(\frac{9}{410}\right)^2(1-\sigma^2)$. For any $X^{(0)}\in \calP_n$, denote $X^{(N)} $ the output of \ppm\ with input ($\calH(A),\calH(B), X^{(0)}, N=2\log n$), where $\calH(M)$ corresponds to the matrix $M$ with the diagonal removed. Then, there exists constants $C',c>0$ such that, for all $n\geq \frac{C'}{\kappa}\log{n}$, it holds \[ \prob\left(\forall X^{(0)}\in \calP_n \text{ such that }|S_{X^{(0)}}|\geq (1-\kappa)n, X^{(N)} = \id\right) \geq \revb{1-e^{-c\kappa n}-3n^{-2}}.\] }%
%
\end{theorem}
The diagonal of the adjacency matrices $A$ and $B$ in Algorithm \ref{alg:ppmgm} was removed in the above theorem only for ease of analysis. Its proof is detailed in Section \ref{subsec:proof_unif_seed_ppm}.
\revb{A direct consequence of Theorem \ref{thm:unif_rec_ppm} is when the seed $X^{(0)}$ is data dependent, i.e., depends on $A,B$. In this case, denoting $\mathcal{E}_0 = \{|S_{X^{(0)}}|\geq (1-\kappa)n\}$ to be the event that $X^{(0)}$ satisfies the requirement of Theorem \ref{thm:unif_rec_ppm}, and $\mathcal{E}_{unif}$ to be the ``uniform'' event in Theorem \ref{thm:unif_rec_ppm}, we clearly have by a union bound that 
\begin{equation*}
    \prob(X^{(N)} = \id) \geq \prob(\mathcal{E}_0) - \prob(\mathcal{E}^c_{unif}) \geq \prob(\mathcal{E}_0) - e^{-c\kappa n}-3n^{-2}. 
\end{equation*}
Hence if $\mathcal{E}_0$ holds with high probability, then exact recovery of $X^* = \id$ is guaranteed with high probability as well.
}
\begin{remark}
Contrary to our previous theorems, here the strong consistency of the estimator holds uniformly over all possible seeds that satisfy the condition $|S_{X^{(0)}}|\geq (1-\kappa)n$. For this reason, we need a stronger condition than $|S_{X^{(0)}}|=\Omega(\sqrt{n\log n})$ as was the case in Theorem \ref{prop:one_it_conv}. Our result is non trivial and cannot be obtained from Theorem  \ref{prop:one_it_conv} by taking a union bound. The proof relies on a decoupling technique adapted from \citep{MaoRud} that used a similar refinement method for CER graphs.
\end{remark}

\begin{remark}
Contrary to the results obtained in the seedless case that require $\sigma=o(1)$ for exact recovery \citep{Grampa}, we can allow $\sigma$ to be of constant order.
The condition the fraction of fixed points in the seed be at least \revb{$1-\kappa = 1-\left(\frac{9}{410}\right)^2(1-\sigma^2)$ } seems to be far from optimal as shown in the experiments in Section \ref{sec:experiments}, \revb{see Fig. \ref{fig:perf-1}}.
But interestingly, this condition shows that when the noise $\sigma$ increases, \ppm\, needs a more accurate initialization to recover the latent permutation. This is confirmed by our experiments.
\end{remark}

\section{Proof outline}
\subsection{Proof of Theorem \ref{prop:one_it_conv}}\label{sec:thm_one_it}
For $A,B\sim W(n,\sigma,\operatorname{id})$, the proof of Theorem \ref{prop:one_it_conv} relies heavily on the concentration properties of the entries of the matrix $C=AXB$, which is the matrix that is projected by our proposed algorithm. In particular, we use the fact that $C$ is diagonally dominant with high probability, under the assumptions of Theorem \ref{prop:one_it_conv}, which is given by the following result. Its proof is delayed to Appendix 
\ref{app:concentration}.
\begin{proposition}
[Diagonal dominance property for the matrix $C=AXB$]\label{prop:diago_dom}
Let $A,B\sim W(n,\sigma,\operatorname{id})$ with correlation parameter $\sigma\in[0,1)$ and let $X\in \calP_n$ with $S_X$ the set of its fixed points and $s_x:=|S_X|/n$. Assume that $s_x\geq 10/n$ and that $n\geq 10$. Then the following is true.
\begin{enumerate}[(i)]
\item \textbf{Noiseless case.} For a fixed $i\in[n]$ it holds that
\begin{equation*}
   \prob\big(\exists j\neq i: (AXA)_{ij}>(AXA)_{ii}\big)\leq 4ne^{-\frac{s_x^2}{96}n}.  
\end{equation*}

\item For $C=AXB$ and $i\in [n]$ it holds \begin{equation*}
    \prob{(\exists j\neq i : C_{ij}>C_{ii})}\leq 5ne^{-c(\sigma)s_x^2n}
\end{equation*}
where $c(\sigma)=\frac1{384}(\frac{1-\sigma^2}{1+2\sigma\sqrt{1-\sigma^2}})$. \end{enumerate}
\end{proposition}
With this we can proceed with the proof of Theorem \ref{prop:one_it_conv}.
\begin{proof}[Proof of Theorem \ref{prop:one_it_conv}]
To prove  part $(i)$ of the theorem it suffices to notice that in Proposition \ref{prop:diago_dom} part $(ii)$ we upper bound the probability that $C=AXB$ is not diagonally dominant for each fixed row. Using the union bound, summing over the $n$ rows, we obtain the desired upper bound on the probability that $C$
is not diagonally dominant.  We now prove part $(ii)$. Notice that
the assumption $\|X-\id\|_F\leq \theta\sqrt{n}$ for $\theta< \sqrt{2}$ implies that $s_x$ is strictly positive. Moreover, from this assumption and the fact that $\|X-\id\|^2_F=2(n-|S_X|)$ we deduce that \begin{equation}\label{eq:theta_fp}
s_x\geq \Big(1-\frac{{\theta}^2}2\Big).
\end{equation}
On the other hand, we have
\begin{align*}
    \prob(\Pi\neq \operatorname{Id})&\leq \prob(C \text{ is not diag.dom})\\
    &= \prob(\exists i,j\in[n],i\neq j:C_{ii}<C_{ij})\\
    &\leq 5n^2e^{-c(\sigma)s_x^2n}\\
    &\leq 5n^2e^{-c(\sigma)\big(1-\frac{\theta^2}2\big)^2n}
\end{align*}
where we used Lemma \ref{lem:diagdom_LAP} in the first inequality, Proposition \ref{prop:diago_dom} in the penultimate step and, \eqref{eq:theta_fp} in the last inequality. 
\end{proof}

\subsubsection{Proof of Proposition \ref{prop:diago_dom}}
In Proposition \ref{prop:diago_dom} part $(i)$ we assume that $\sigma=0$. The following are the main steps of the proof.
\begin{enumerate}
\item We first prove that for all $X\in\calP_n$ such that $s_x=|S_X|/n$ and for $i\neq j\in[n]$ the gap $C_{ii}-C_{ij}$ is of order $s_x$ in expectation. 
\item We prove that $C_{ii}$ and $C_{ij}$ are sufficiently concentrated around its mean. In particular, the probability that $C_{ii}$ is smaller than $s_x/2$ is exponentially small. The same is true for the probability that $C_{ij}$ is larger than $s_x/2$.
\item We use the fact $\prob(C_{ii}\leq C_{ij})<\prob(C_{ii} \leq s_x/2)+\prob(C_{ij}\geq s_x/2)$  to control the probability that $C$ is not diagonally dominant. 
\end{enumerate}

The proof is mainly based upon the following two lemmas.
\begin{lemma}\label{lem:expectation}
For the matrix $C=AXA$ and with $s_x=|S_X|/n$ we have 
\[\expec[C_{ij}]=\begin{cases}
s_x+\frac1n\mathbbm{1}_{i\in S_X}  \enskip\text { for }i=j, \\
\frac1n\mathbbm{1}_{x(j)=i} \enskip\text { for }i\neq j, \\
\end{cases}\]
and from this we deduce that for $i,j\in[n]$ with $i\neq j$ \[s_x-\frac1n\leq \expec{[C_{ii}]}-\expec{[C_{ij}]}\leq s_x+\frac1n.\]
\end{lemma}

\begin{lemma}\label{lem:tailbounds} 
Assume that $s_x\in(10/n,1]$ and $n\geq 10$. Then for $i,j\in[n]$ with $i\neq j$ we have
\begin{align}\label{eq:bounddiag}
    \prob(C_{ii}\leq s_x/2)&\leq 4 e^{-\frac{s_x^2}{48}n}, \\
    \label{eq:boundoffdiag}
    \prob(C_{ij}\geq s_x/2)&\leq 3e^{-\frac{s_x^2}{96}n}. 
\end{align}
\end{lemma}
With this we can prove Proposition \ref{prop:diago_dom} part $(i)$.

\begin{proof}[Proof of Prop. \ref{prop:diago_dom} $(i)$]
Define the event $\mathcal{E}_j=\{C_{ii}<\frac{s_x}2\}\cup \{C_{ij}>\frac{s_x}2\}$ and note that for $j\neq i$, we have $\{C_{ij}>C_{ii}\}\subset\mathcal{E}_j$. With this and the bounds \eqref{eq:bounddiag} and \eqref{eq:boundoffdiag} we have 
\begin{align*}
    \prob\big(\exists j\neq i: C_{ij}>C_{ii}\big)&=\prob(\cup_{j\neq i}\{C_{ij}>C_{ii}\})\\
    &\leq \prob(\cup_{j\neq i}\mathcal{E}_j)\\
    &\leq  \prob(C_{ii}\leq \frac{s_x}{2})+\sum_{j\neq i} \prob(C_{ij}\geq \frac{s_x}{2})\\
    &\leq 4e^{-\frac{s_x^2}{96}n}+3(n-1)e^{-\frac{s_x^2}{96}n}\\
    &\leq 4ne^{-\frac{s_x^2}{96}n}.
\end{align*}
\end{proof}
The proof of Lemma \ref{lem:expectation} is short and we include it in the main body of the paper. On the other hand, the proof of Lemma \ref{lem:tailbounds} mainly uses concentration inequalities for Gaussian quadratic forms, but the details are quite technical. Hence we delay its proof to Appendix \ref{app:diagdom_row_noiseless}. Before proceeding with the proof of Lemma \ref{lem:expectation}, observe that the following decomposition holds for the matrix $C$.
\begin{equation}\label{eq:Cdecom}
C_{ij}=\sum_{k,k'}A_{ik}X_{k,k'}A_{k'i} 
= \begin{cases}
\sum_{k\in S_X}A^2_{ik}+\sum_{k\notin S_X}A_{ik}A_{ix(k)} \enskip\text { for }i=j,\\
\sum^n_{k=1}A_{ik}A_{x(k)j} \enskip\text{ for }i\neq j.
\end{cases}
\end{equation}
\begin{proof}[Proof of Lemma \ref{lem:expectation}]
From \eqref{eq:Cdecom} we have that 
\begin{align*}
    \expec[C_{ii}]
    =\sum_{k\in S_X}\expec[A^2_{ik}]+\sum_{k\notin S_X}\expec[A^2_{ik}]
    =\frac{|S_X|}n+\frac{\mathbbm{1}_{i\in S_X}}n.
\end{align*}
Similarly, for $j\neq i$ it holds
\begin{align*}
    \expec[C_{ij}] =\sum^n_{k=1}\expec[A_{ik}A_{x(k)j}]
    =\frac1n\mathbbm{1}_{i,j\notin S_X, x(j)=i}
    =\frac{\mathbbm{1}_{x(j)=i}}n
\end{align*}
from which the results follows easily.
\end{proof}

The proof of Proposition \ref{prop:diago_dom} part $(ii)$ which corresponds to the case $\sigma\neq 0$ uses similar ideas and the details can be found Appendix \ref{app:diagdom_row_noise}. 

\subsection{Proof of Theorem \ref{prop:partial_rec}} \label{subsec:proof_thm_partial_rec}
The proof of Theorem \ref{prop:partial_rec} will be based on the following lemma, which extends Proposition \ref{prop:diago_dom}. 

\begin{lemma}\label{lem:not_rc_dom}
For a fixed $i\in[n]$, we have 
\begin{equation*}
  \prob(C_{ii} \textup{ is not row-column dominant})\leq 16ne^{-c(\sigma)s_x^2n}.
\end{equation*}
\end{lemma}
The proof of Lemma \ref{lem:not_rc_dom} is included in Appendix \ref{app:lem_not_rc_dom}. We now prove Theorem \ref{prop:partial_rec}. 
The main idea is that for a fixed $i\in[n]$, with high probability the term $C_{ii}$ will be the largest term in the $i$-th row and the $i$-th column, and so \texttt{GMWM} will assign $\pi(i)=i$. We will also use the following event inclusion, which is direct from \eqref{eq:overlap_event} in Lemma \ref{lem:overlap_event}.

\begin{equation}
    \label{eq:overlap_event2}
    \{\operatorname{overlap}(\pi,\operatorname{id})< r/n\}\subset\bigcup^r_{i=1}\{C_{ii} \text{ is not row-column dominant }\}.
\end{equation}
\begin{proof}[Proof of Theorem \ref{prop:partial_rec} ]
Fix $i\in[n]$. By  \eqref{eq:overlap_event2} we have that 
\begin{align*}
    \prob(\operatorname{overlap}(\pi,\operatorname{id})\leq r/n)&\leq \sum^r_{i=1}\prob(C_{ii} \text{ is not row-column dominant})\\
    &\leq \sum^r_{i=1}\prob(\exists j\neq i,\text{ s.t }C_{ij}\vee C_{ji}>C_{ii} )\\
    &\leq 16rne^{-c(\sigma) s_x^2n}
\end{align*}
where we used Lemma \ref{lem:not_rc_dom} in the last inequality.
\end{proof}

\begin{remark} Notice that the RHS of \eqref{eq:overlap_event2} is a superset of the RHS of \eqref{eq:overlap_event}. To improve this, it is necessary to include dependency information. In other words, we need to `beat Hölder's inequality'. To see this, define \[
E_i:=\mathbbm{1}_{C_{ii}\text{ is not row-column dominant }},\enskip \varepsilon_{I}:=\mathbbm{1}_{\sum_{i\in I}E_i>0}, \text{ for } I\subset [n];
\] 
then $\varepsilon_{I'}$, for $I'=[r]$, is the indicator of the event in the RHS of \eqref{eq:overlap_event2}. On other hand, the indicator of the event in the RHS of \eqref{eq:overlap_event} is ${\displaystyle \prod_{\substack{I\subset[n],|I|=r}}}\varepsilon_I$. If $\expec\big[\varepsilon_I\big]$ is equal for all $I$, then Hölder inequality gives \[\expec\Big[{\displaystyle \prod_{\substack{I\subset[n],|I|=r}}}\varepsilon_I\Big]\leq \expec[\varepsilon_{I'}]\] which does not help in quantifying the difference between \eqref{eq:overlap_event} and \eqref{eq:overlap_event2}. This is not surprising as we are not taking into account the dependency between the events $\varepsilon_I$ for the different sets $I\subset[n],|I|=r$.
\end{remark}

\subsection{Proof of Theorem \ref{thm:unif_rec_ppm}} \label{subsec:proof_unif_seed_ppm}
The general proof idea is based on the decoupling strategy used by \citep{MaoRud} for \erdos-R\'enyi graphs. To extend their result from binary graphs to weighted graphs, we need to use an appropriate measure of similarity. For $i, i'\in [n], W\subset [n]$ and $g\in \calS_n$, let us define 
\[ 
\langle A_{i:}, B_{i':} \rangle_{g,W} := \sum_{j\in W} A_{ig(j)}B_{i'j}
\] 
to be the similarity between $i$ and $i'$ restricted to $W$ and measured with a scalar product depending on $g$ (the permutation used to align $A$ and $B$). When $g=\operatorname{id}$ or $W=[n]$ we will drop the corresponding subscript(s). If $A$ and $B$ were binary matrices, we would have the following correspondence  \[ \langle A_{i:}, B_{i':} \rangle_{g,W} = |g(\calN_A(i)\cap W)\cap \calN_B(i') |.\]  This last quantity plays an essential role in Proposition 7.5 of \citep{MaoRud}. Here $g(\calS)$ denotes the image of a set $\calS \subseteq [n]$ under permutation $g$, and $\mathcal{N}_A(i)$ represents the set of neighboring vertices of $i$ in the graph $A$. This new measure of similarity has two main implications on the proof techniques, marking a departure from the work in \cite{MaoRud}. Firstly, one can no longer rely on the fact that $\langle A_{i:}, B_{i':} \rangle_{g,W}$ is non-negative. Secondly, we will need different concentration inequalities to handle these quantities.

\paragraph{ Step 1.} The algorithm design relies on the fact that if the matrices $A$ and $B$ were correctly aligned then the correlation between $A_{i:}$ and $B_{i:}$ should be large and the correlation between $A_{i:}$ and $B_{i':}$ should be small for all $i\neq i'$. The following two lemmas precisely quantify these correlations when the two matrices are well aligned.

\begin{lemma}[Correlation between corresponding nodes]\label{lem:nb_ngbh1_mt}
Let $(A,B)\sim W(n,\sigma, x^*=\operatorname{id})$ and assume that the diagonals of $A$ and $B$ have been removed. Then for $n$ large enough \revb{(larger than a constant)}, we have with probability at least $1-n^{-2}$ that
\[ \langle A_{i:}, B_{i:}\rangle \geq \sqrt{1-\sigma^2}(1-\epsilon_1)-\sigma \epsilon_2 \text{ for all } i\in [n], \]
where \revb{$0 < \epsilon_1, \epsilon_2 \leq  C\sqrt{\frac{\log n}{n}}$ for a constant $C>0$}.
\end{lemma}

\begin{lemma}[Correlation between different nodes]\label{lem:nb_ngbh2_mt}
Let $(A,B)\sim W(n,\sigma, \operatorname{id})$ and assume that the diagonals of $A$ and $B$ have been removed. Then for $n$ large enough \revb{(larger than a constant)}, we have with probability at least $1-n^{-2}$ that
\[ \left| \langle A_{i:}, B_{i':} \rangle\right|\leq \sqrt{1-\sigma^2}\revb{\epsilon_3}+\sigma \revb{\epsilon_4} \text{ for all } i,i'\in [n] \text{ such that } i'\neq i, \]
where \revb{$0 < \epsilon_3, \epsilon_4 \leq  C\sqrt{\frac{\log n}{n}}$ for a constant $C>0$}.
\end{lemma}
The proofs of Lemma's \ref{lem:nb_ngbh1_mt} and \ref{lem:nb_ngbh2_mt} can be found in Appendix \ref{sec:app_thm3}. 

\paragraph{Step 2.} Since the ground truth alignment between $A$ and $B$ is unknown, we need to use an approximate alignment (provided by $X^{(0)}$). It will suffice that $X^{(0)}$ is close enough to the ground truth permutation. This is linked to the fact that if $|S_{X^{(0)}}|$ is large enough then the number of nodes for which there is a substantial amount of information contained in ${S^c_{X^{(0)}}}$ is small. This is shown in the following lemma.

\begin{lemma}[Growing a subset of vertices]\label{lem:growing_vert}
%
%
Let $G$ a graph generated from the \revb{Gaussian} Wigner model with self-loops removed, associated with an adjacency matrix $A$. \revb{ For any $I\subseteq [n]$ and $\kappa\in(0,\frac12)$, define the random set
\[ \tilde{I}= \lbrace i \in [n]: \norm{A_{i:}}_{I^c}^2<8\kappa \rbrace .\]
}
Then for \revb{$n\geq \frac{C'}{\kappa}\log{n}$ where $C'>0$ is a large enough constant}, we have
\[\prob \left(\revb{\forall I\subseteq [n] \text{ with } |I|\geq (1-\kappa)n, \text{ it holds } } |\tilde{I}^c|\leq \frac{1}{4}|I^c| \right) \geq 1-e^{-c' \kappa n} \] for some constant $c' > 0$.
\end{lemma}

In order to prove this lemma we will need the following decoupling lemma.
\begin{lemma}[An elementary decoupling] \label{lem:decoupling} Let $M>0$ be a parameter and $G$ be a weighted graph on $[n]$, with weights of magnitude bounded by $1$ and without self loops, represented by an adjacency matrix $A\in [-1,1]^{n\times n}$. Assume that there are two subsets of vertices $Q,W\subset [n]$ such that
\[ \norm{A_{i:}}_W^2 \geq M \text{ for all } i\in Q.\]
Then there are subsets $Q'\subseteq Q$ and $W'\subseteq W$ such that $Q'\cap W' =\emptyset$, $|Q'|\geq |Q|/5$ and 
\[ \norm{A_{i:}}_{W'}^2 \geq M/2 \text{ for all } i\in Q'. \]
\end{lemma}

\begin{proof}
 If $|Q\setminus W|\geq |Q|/5$ then one can take $Q'=Q\setminus W$ and $W'= W$. So we can assume that $|Q\cap W|\geq 4|Q|/5$. Let $\Tilde{W}:=W\setminus Q$ and $\hat{Q}$ be a random  subset of $Q\cap W$ where each element $j\in Q\cap W$ is selected independently with probability $1/2$ in $\hat{Q}$. Consider the random disjoint sets $\hat{Q}$ and $W':=\tilde{W}\cup ((Q\cap W)\setminus \hat{Q})$. First, we will show the following claim.
\begin{claim}
For every $i \in Q\cap W$, we have 
$ \prob( \norm{A_{i:}}_{W'}^2\geq M/2 |i \in \hat{Q})\geq 1/2.$
\end{claim}
Indeed, we have by definition
\[  \norm{A_{i:}}_{W'}^2=\sum_{j \in W' } A_{ij}^2 = \sum_{j \in  W\cap Q } A_{ij}^2\indic_{j\not \in \hat{Q}} +\sum_{j \in  \tilde{W} } A_{ij}^2 .\]
By taking the expectation conditional on $i \in \hat{Q}$, we obtain
\[ \expec \left( \norm{A_{i:}}_{W'}^2 \middle| i \in \hat{Q} \right)
= \sum_{j \in  W\cap Q} \frac{A_{ij}^2}{2} + \sum_{j \in \tilde W} A_{ij}^2
\geq  \frac{1}{2}\sum_{j\in W}A_{ij}^2 \geq \frac{M}{2}.\]
But since $\sum_{j\in W\cap Q} A_{ij}^2(\indic_{j\not \in \hat{Q}}-\frac{1}{2})$ is a symmetric random variable we have that 
\[ \prob\left(\norm{A_{i:}}_{W'}^2\geq \expec(\norm{A_{i:}}_{W'}^2) \middle|i \in \hat{Q}\right) = 1/2\]
and hence 
\[ \prob\left(\norm{A_{i:}}_{W'}^2\geq \frac{M}{2} \middle|i \in \hat{Q}\right)\geq 1/2.
\]
Consequently, we have
\[ \expec\left(\sum_{i\in Q\cap W} \indic_{\lbrace \norm{A_{i:}}_{W'}^2\geq M/2 \rbrace} \indic_{i \in \hat{Q}}\right) = \sum_{i\in Q\cap W} \prob(i\in \hat{Q})\expec\left( \indic_{\lbrace \norm{A_{i:}}_{W'}^2\geq M/2 \rbrace}\middle| i\in \hat{Q}\right) \geq \frac{|Q\cap W|}{4} \geq \frac{|Q|}{5}.\]
Therefore, there is a realization $Q'$ of $\hat{Q}$ such that $Q'$ and $W'$ satisfy the required conditions.
\end{proof}

\begin{proof}[Proof of Lemma \ref{lem:growing_vert}]
\revb{Let us define $\delta:=8\kappa$, which will be used throughtout this proof.} By considering sets $W=I^c$ and $Q \subseteq \tilde{I}^c$ we obtain the following inclusion 
 \begin{align*}
     \lbrace & \revb{\exists I\subseteq [n] \text{ with } |I|\geq (1-\kappa)n,\text{ such that } }|\tilde{I}^c|> \frac{1}{4}|I^c|  \rbrace \subset \\
     &\calE:= \lbrace \exists\, Q,W\subseteq [n]: |W|\leq \kappa n, |Q|\geq |W|/4\neq 0, \norm{A_{i:}}_{W}^2\geq \delta \text{ for all }i\in Q  \rbrace . 
 \end{align*} 
 According to Lemma \ref{lem:decoupling}, $\calE$ is contained in
 \[\calE':=\lbrace  \exists\, Q', W'\subseteq [n]: |W'|\leq \kappa n, |Q'|\geq |W|/20\neq 0, Q'\cap W'= \emptyset,  \norm{A_{i:}}_{W'}^2\geq \delta/2 \text{ for all }i\in Q' \rbrace.\]
 For given subsets $Q'$ and $W'$, the random variables $(\norm{A_{i:}}_{W'}^2)_{i\in Q'}$ are independent. So, by a union bound argument we get 
 \begin{align*} \prob \left( \revb{\exists I\subseteq [n] \text{ with } |I|\geq (1-\kappa)n,\text{ such that } } |\tilde{I}^c|> \frac{1}{4}|I^c|  \right) &\leq\\
 \sum_{w=1}^{\ceil{\kappa n}}\sum_{|W'|=w}\sum_{k=\ceil{w/20}}^{n}\binom{n}{k}&\prob\left( \norm{A_{i:}}_{W'}^2\geq \delta/2 \right)^k. 
 \end{align*} 
 According to Lemma \ref{lem:lau_mass}, for the choice $t=\kappa n$  we have for all $W'$
\[ \prob\left( \norm{A_{i:}}_{W'}^2\geq \delta/2 \right) \leq  \prob\left( n\norm{A_{i:}}_{W'}^2\geq |W|+\sqrt{|W|t}+2t \right) \leq e^{-\kappa n}.\] 

Consequently, for $n$ large enough, we have 
\begin{align*} \prob \left( \revb{\exists I\subseteq [n] \text{ with } |I|\geq (1-\kappa)n,\text{ such that } } |\tilde{I}^c|> \frac{1}{4}|I^c|  \right) &\leq\\
\sum_{w=1}^{\ceil{\kappa n}}\sum_{k=\ceil{w/20}}^{n}&\left(\frac{en}{w}\right)^w\left(\frac{en}{k}\right)^ke^{-k\kappa n} < e^{-c\kappa n},\end{align*} for a constant $c>0$.
Indeed, since \[ \frac{en}{ke^{\kappa n}}<1\] \revb{because by assumption $n\geq \frac{C'}{\kappa}\log{n}$ , we obtain}
\[ \sum_{k=\ceil{w/20}}^{n} \left(\frac{en}{k}\right)^ke^{-k\kappa n} \leq C\left(\frac{en}{e^{\kappa n}}\right)^{\ceil{w/20}}\]
by the property of geometric series, where $C>0$ is a constant. But by the same argument 
\[ \sum_{w=1}^{\ceil{\kappa n}}\left(\frac{en}{w}\right)^w \left(\frac{(en)^{1/20}}{e^{\kappa n/20}}\right)^{w} \leq \frac{(en)^{1/20}}{e^{\kappa n/20}}\leq e^{-c\kappa n}\]
where $c > 0$ is a constant.
\end{proof}

\paragraph{Step 3.} We are now in position to show that at each step the set of fixed points of the permutation obtained with \ppm\, increases.

\begin{lemma}[Improving a partial matching]\label{lem:improve_matching}
\revb{Let $(A, B) \sim W(n, \sigma,\operatorname{id})$ with $\sigma\in[0,1)$, and define \[ \kappa := \left(\frac{9}{410}\right)^2(1-\sigma^2).\]For any $g\in\calS_n $, let us denote $\Tilde{g}$ as the output of one iteration of \ppm\, with $g$ as an input. Then, there exist constants $C',c>0$, such that the following is true. If $\frac{n}{\log{n}}\geq \frac{C'}{\kappa}$, then with probability at least $1-e^{-c\kappa n}-3n^{-2}$, it holds for all $g \in \calS_n$ satisfying $ |i\in [n]: g(i)=i |\geq (1-\kappa)n $ that \[ |\lbrace i\in[n]: \tilde{g}(i)=i \rbrace|\geq \frac{n}{2}+\frac{|\lbrace i\in[n]: g(i)=i \rbrace|}{2}.\]
}
\end{lemma}


\begin{proof}
\revb{
For any $I\subseteq [n]$ such that $|I|\geq (1-\kappa)n$, define 
\begin{align*}
    \tilde{I}:=&\lbrace i\in [n]: \norm{A_{i:}}_{I^c}^2 < 8\kappa \rbrace ,\\
    \tilde{I}':=&\lbrace i\in [n]: \norm{B_{i:}}_{I^c}^2 < 8\kappa \rbrace .
\end{align*}
Consider the events
\begin{align*} 
\calE_1' &:=\lbrace \forall I\subseteq [n] \text{ s.t. } |I|\geq (1-\kappa)n, \text{ it holds that } |\tilde{I}^c|\vee |(\tilde{I}')^c|\leq \frac{1}{4}|I^c|  \rbrace , \\
\calE_2' &:=\lbrace \forall i\in [n] : \ \langle A_{i:}, B_{i:}\rangle \geq 0.9\sqrt{1-\sigma^2} \rbrace , \\
\calE_3' &:=\lbrace \forall i\neq i'\in [n]: \  |\langle A_{i:}, B_{i':}\rangle|<C\sqrt{\log n/n} \rbrace ,\\
\calE_4' &:=\lbrace \forall i\in [n]: \  \norm{A_{i:}}, \norm{B_{i:}}<2 \rbrace,
\end{align*}
where $C$ is the constant appearing in Lemmas \ref{lem:nb_ngbh1_mt} and \ref{lem:nb_ngbh2_mt}.}
\revb{
Define $\calE'=\cap_{i=1}^4 \calE_i'$. By Lemmas~\ref{lem:growing_vert}, \ref{lem:nb_ngbh1_mt}, \ref{lem:nb_ngbh2_mt} and \ref{lem:lau_mass} we have \[ \prob(\calE')\geq 1-e^{-c\kappa n}-3n^{-2}. \]}
\revb{Now condition on $\calE'$ and take any $g\in \calS_n$ such that $|\lbrace i\in [n]: g(i)=i\rbrace|\geq (1-\kappa)n$. Define $I$ as the set of fixed points of $g$.}
\revb{For all $i \in \tilde{I}\cap \tilde{I}' $ the following holds.
\begin{enumerate}
    \item We have \begin{align*}
        \langle A_{i:}, B_{i:}\rangle_g &\geq \langle A_{i:}, B_{i:}\rangle -|\langle A_{i:}, B_{i:}\rangle_{g,I^c}| -|\langle A_{i:}, B_{i:}\rangle_{I^c}|\\
        &\geq 0.9\sqrt{1-\sigma^2}- 2\norm{A_{i:}}_{I^c}\norm{B_{i':}}_{I^c} \tag{by $\calE'_3$, Cauchy-Schwartz and the fact that $g(I^c)=I^c$}\\
        &\geq 0.9\sqrt{1-\sigma^2}-16\kappa
    \end{align*}
    \item For all $i'\neq i$ \begin{align*}
        \langle A_{i:}, B_{i':}\rangle_g &\leq |\langle A_{i:}, B_{i':}\rangle|+ |\langle A_{i:}, B_{i':}\rangle_{I^c}|+|\langle A_{i:}, B_{i':}\rangle_{g,I^c}|\\
        &\leq C\sqrt{\frac{\log n}{n}}+ 2\norm{A_{i:}}_{I^c}\norm{B_{i':}}_{I^c} \tag{by $\calE'_3$, Cauchy-Schwartz and the fact that $g(I^c)=I^c$}\\
        &\leq C\sqrt{\frac{\log n}{n}}+12\sqrt{\kappa}
    \end{align*}
     since $i\in \tilde{I}\cap \tilde{I}'$ and $\norm{B_{i':}}_{I^c}\leq \norm{B_{i':}}\leq 2$ on the event $\calE_4'$.
     \item Arguing in the same way, we have for all $i'\neq i$ \[
      \langle A_{i':}, B_{i:}\rangle_g \leq C\sqrt{\frac{\log n}{n}}+12\sqrt{\kappa}.
     \]
\end{enumerate}
}
\revb{Since the condition on $n$ implies  $C\sqrt{\frac{\log n}{n}}\leq 12\sqrt{\kappa}$, and $\kappa\in(0,1)$, it follows that for all $i \in \tilde{I}\cap \tilde{I}' $ and $i'\neq i$ 
\begin{align*}
    \langle A_{i:}, B_{i:}\rangle_g &\geq 0.9\sqrt{1-\sigma^2}-16\sqrt{\kappa},\\
    \langle A_{i:}, B_{i':}\rangle_g   \vee\langle A_{i':}, B_{i:}\rangle_g&\leq 25\sqrt{\kappa}.
\end{align*} 
Thus, from the stated choice of $\kappa$, we have for all $i \in \tilde{I}\cap \tilde{I}' $ and $i'\neq i$ that \[
\langle A_{i:}, B_{i:}\rangle_g \geq  \langle A_{i:}, B_{i':}\rangle_g \vee  \langle A_{i':}, B_{i:}\rangle_g.
\]
In particular, this implies that for all $i \in \tilde{I}\cap \tilde{I}' $, $(AGB)_{ii}$ is row-column dominant, where $G\in \calP_n$ corresponds to $g\in \calS_n$. Hence, $\tilde{g}$ (the output of $\ppm$) will correctly match $i$ with itself. }
\revb{Then clearly $\tilde{g}$ will be such that \begin{align*}
    |\lbrace i\in [n]: \tilde{g}(i)=i \rbrace |&\geq |\tilde{I}\cap \tilde{I}'| \tag{by inclusion}\\
    &\geq n-|\tilde{I}^c|-|(\tilde{I}')^c|\tag{by union bound}\\
    &\geq \frac{n}{2}+\frac{|\lbrace i\in [n]: g(i)=i \rbrace|}{2}.\tag{by $\calE'_1$}
\end{align*}}
\end{proof}
\paragraph{Conclusion.} By Lemma \ref{lem:improve_matching}, if the initial number of fixed points is $(1-\kappa)n$ then after one iteration step the size of the set of fixed points of the new iteration is at least $(1-\kappa/2)n$ with probability greater than $1-e^{-c\kappa n}-3n^{-2}$. So after $2\log n$ iterations the set of fixed points has size at least $(1-\kappa/2^{2\log n})n>n-1$ with probability greater than \revb{$1-e^{-c\kappa n}-3n^{-2}$. Note that the high probability comes from conditioning on the event $\calE'$ so it is not necessary to take a union bound for each iteration.} 

\section{Numerical experiments}\label{sec:experiments}
In this section, we present numerical experiments to assess the performance of the \texttt{PPMGM} algorithm and compare it to the state-of-art algorithms for graph matching, under the CGW model\footnote{All the experiments were conducted using MATLAB R2021a (MathWorks Inc., Natick, MA). The code is available at \url{https://github.com/ErnestoArayaV/Graph-matching-PPMGM}.}. We divide this section in two parts. In Section \ref{sec:perf_comp} we generate CGW graphs $A,B\sim W(n,\sigma,x^*)$ for a uniformly random permutation $x^*$, and apply to $A,B$ the spectral algorithms \texttt{Grampa} \citep{Grampa}, the classic \texttt{Umeyama} \citep{Spectral_weighted_Ume}, and the convex relaxation algorithm \texttt{QPADMM}\footnote{This algorithm is refered to as $\texttt{QP-DS}$ in \citep{Grampa}. Since algorithm $\texttt{ADMM}$ is used to obtain a solution, we opt to use the name $\texttt{QPADMM}$.}, which first solves the following convex quadratic programming problem 
\begin{equation}\label{eq:exp_convex}
    \max_{X\in \calB_n}\|AX-XB\|^2_F,
\end{equation}
where $\calB_n$ is the Birkhoff polytope of doubly stochastic matrices, and then rounds the solution using the greedy method $\texttt{GMWM}$. All of the previous algorithms work in the seedless case. As a second step, we apply algorithm \texttt{PPMGM} with the initialization given by the output of \texttt{Grampa}, \texttt{Umeyama} and $\texttt{QPADMM}$. We show experimentally that by applying \texttt{PPMGM} the solution obtained in both cases improves, when measured as the overlap (defined in \eqref{eq:overlap_def}) of the output with the ground truth. We also run experiments by initializing \texttt{PPMGM} with $X^{(0)}$ randomly chosen at a certain distance of the ground truth permutation $X^*$. Specifically, we select $X^{(0)}$ uniformly at random from the set of permutation matrices that satisfy $\|X^{(0)}-X^*\|_F=\theta'\sqrt{n}$, and vary the value of $\theta'\in [0,\sqrt{2})$.

In Section \ref{sec:spar_st} we run algorithm \texttt{PPMGM} with different pairs of input matrices. We consider the Wigner correlated matrices $A,B$ and also the pairs of matrices ($A^{\operatorname{spar}_1},B^{\operatorname{spar}_1}$), ($A^{\operatorname{spar}_2},B^{\operatorname{spar}_2}$) and ($A^{\operatorname{spar}_3},B^{\operatorname{spar}_3}$), which are produced from $A,B$ by means of a sparsification procedure (detailed in Section \ref{sec:spar_st}). The main idea behind this setting is that, to the best of our knowledge, the best theoretical guarantees for exact graph matching have been obtained in \citep{MaoRud} for relatively sparse \erdos-R\'enyi graphs. The algorithm proposed in \citep{MaoRud} has two steps, the first of which is a seedless type algorithm which produces a partially correct matching, that is later refined with a second algorithm \citep[Alg.4]{MaoRud}. Their proposed algorithm \texttt{RefinedMatching} shares similarities with \texttt{PPMGM} and with algorithms \texttt{1-hop} \citep{LubSri,YuXuLin} and \texttt{2-hop} \citep{YuXuLin}. Formulated as it is, \texttt{RefinedMatching} \citep{MaoRud} (and the same is true for \texttt{2-hop} for that matter) only accepts binary edge graphs as input and also uses a threshold-based rounding approach instead of Algorithm \ref{alg:gmwm}, which might be difficult to calibrate in practice. With this we address experimentally the fact that the analysis (and algorithms) in \citep{MaoRud} does not extend automatically to a simple `binarization' of the (dense) Wigner matrices, and that specially in high noise regimes, the sparsification strategies do not perform very well. 

\subsection{Performance of \texttt{PPMGM}}\label{sec:perf_comp}
In Figure \ref{fig1-a} we plot the recovery fraction, which is defined as the overlap (see \eqref{eq:overlap_def}) between the ground truth permutation and the output of five algorithms: \texttt{Grampa}, \texttt{Umeyama}, \texttt{Grampa+PPMGM}, \texttt{Umeyama+PPMGM} and \texttt{PPMGM}. The algorithms \texttt{Grampa+PPMGM} and \texttt{Umeyama+PPMGM} use the output of \texttt{Grampa} and \texttt{Umeyama} as seeds for \texttt{PPMGM}, which is performed with $N=5$. In the algorithm \texttt{PPMGM}, we use an initial permutation $x^{(0)}\in \calS_n$ chosen uniformly at random in the set of permutations such that $\operatorname{overlap}{(x^{(0)},x^*)}=0.1$; this is referred to as `\texttt{PPMGM} rand.init'. We take $n=800$ and plot the average overlap over $25$ Monte Carlo runs. The area comprises $90\%$ of the Monte Carlo runs (leaving out the $5\%$ smaller and the $5\%$ larger). As we can see from this figure, the performance of $\ppm$ initialized with a permutation with $0.1$ overlap with the ground truth  outperforms \texttt{Grampa} and \texttt{Umeyama} (and also their refined versions, where $\ppm$ is used as a post-processing step). Overall, the $\ppm$ improves the performance of those algorithms, provided that their output has a reasonably good recovery. From Fig.\ref{fig1-a}, we see that \texttt{Grampa} and \texttt{Umeyama} fail to provide a permutation with good overlap with the ground truth for larger values of $\sigma$ (for example, at $\sigma=0.5$ both algorithms have a recovery fraction smaller than $0.1$). In Figure \ref{fig1-b} we plot the performance of the \texttt{PPMGM} algorithm for randomly chosen seeds and with different number of correctly pre-matched vertices. More specifically, we consider an initial permutation $x^{(0)}_j\in \calS_n$  (corresponding to initializations $X^{(0)}_j\in\calP_n$) for $j=1,\cdots,6$ with $\operatorname{overlap}(x^{(0)}_1,x^*)=0.04$, $\operatorname{overlap}(x^{(0)}_2,x^*)=0.0425$, $\operatorname{overlap}(x^{(0)}_3,x^*)=0.{045}$, $\operatorname{overlap}(x^{(0)}_4,x^*)=0.05$, $\operatorname{overlap}(x^{(0)}_5,x^*)=0.06$ and $\operatorname{overlap}(x^{(0)}_6,x^*)=0.1$. We call these instances $\emph{in.1,in.2},\ldots,\emph{in.6}$ respectively. Equivalently, these initializations satisfy $\|X^{(0)}_j-X^*\|_F=\theta'_j\sqrt{n}$, where $\theta'_j=\sqrt{2\big(1-\operatorname{overlap}(x^{(0)}_j,x^*)\big)}$. Each permutation $x^{(0)}_j$ is chosen uniformly at random in the subset of permutations that satisfy each overlap condition. We observe that initializing the algorithm with an overlap of $0.1$ with the ground truth permutation already produces perfect recovery in one iteration for levels of noise as high as $\sigma=0.8$. Interestingly, the variance over the Monte Carlo runs diminishes as the overlap with the ground truth increases. In Fig.\ref{fig1-b} the shaded area contains $90\%$ of the Monte Carlo runs, only for $in.4$, $in.5$ and $in.6$ (given the very high variance of the rest, we opt not to share their $90\%$ area for readability purposes). 

\begin{figure}
  \begin{subfigure}[t]{0.475\textwidth}
    \includegraphics[width=\textwidth]{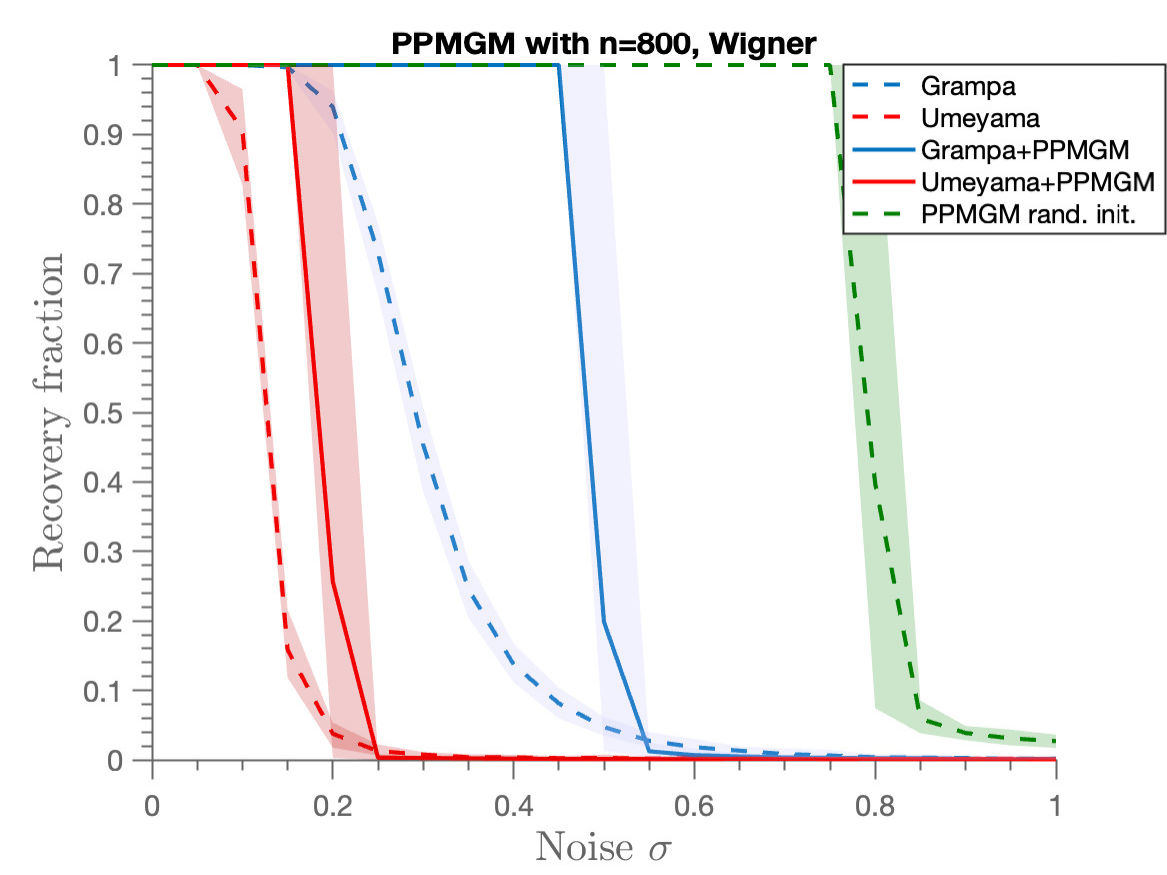}
    \caption{Performance of \texttt{PPMGM} as a refinement of Grampa and Umeyama algorithms, compared with PPM with a random initialization $x^{(0)}$, such that $\operatorname{overlap}(x^{(0)},x^*)=0.1$. The lines represent the average fraction of recovery over $25$ Monte Carlo runs. The shaded area contains $90\%$ of the Monte Carlo runs.}
    \label{fig1-a}
  \end{subfigure}\hfill
  \begin{subfigure}[t]{0.475\textwidth}
    \includegraphics[width=\textwidth]{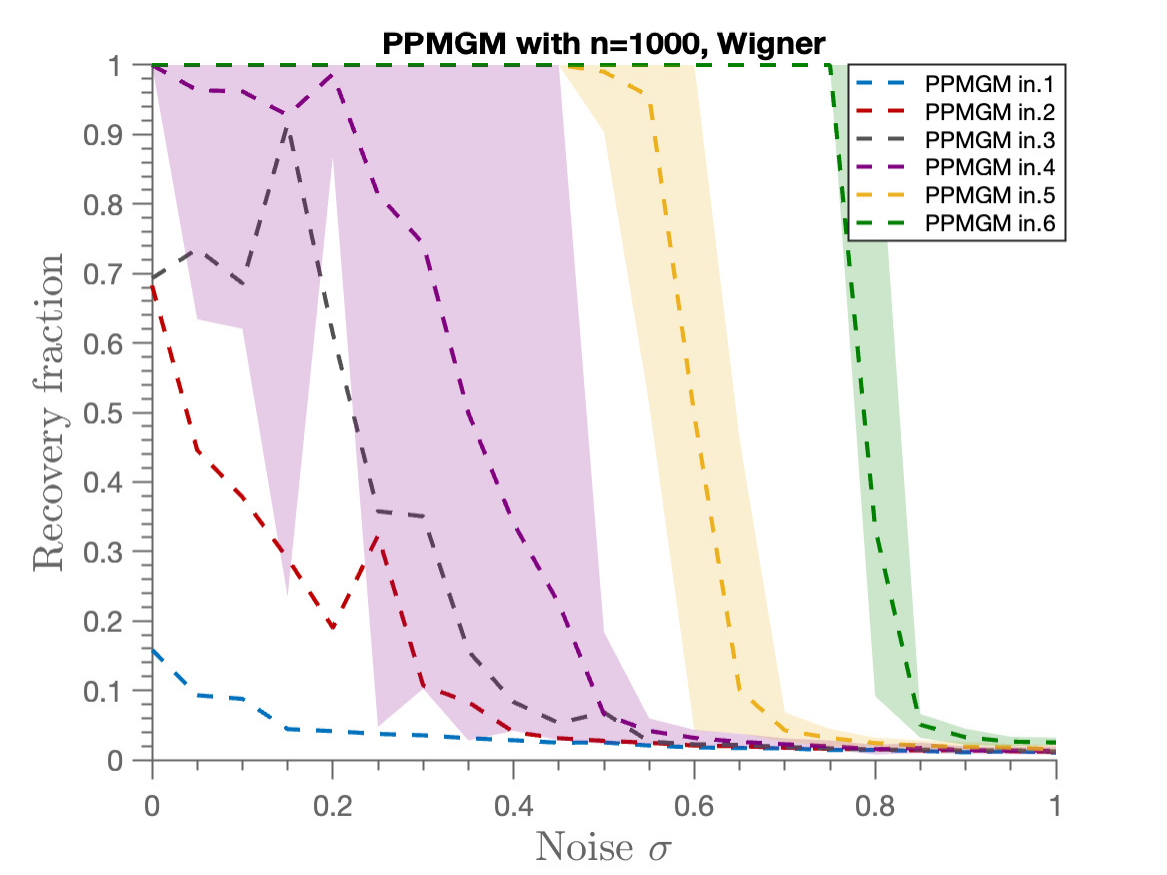}
    \caption{Performance of \texttt{PPMGM} with different initializations. Here $in.1,in.2,in.3,in.4,in.5, in.6$ correspond to respective overlaps of $0.04,0.0425,0.045,0.05,0.06$,$0.1$ of $x^{(0)}$ with the ground truth. We shade the area with $90\%$ of Monte Carlo runs for $in.4,in.5$ and $in.6$ (for the rest, the variance is too high).}
    \label{fig1-b}
  \end{subfigure}
  \caption{We plot the performance of $\ppm$ (with $N=5$)
    as a refinement (post-processing) method of seedless graph matching algorithms, and with random initializations (uniform on different Frobenius spheres ).} 
  \label{fig:perf-1}
\end{figure}

In Figure \ref{fig-convex} we illustrate the performance of $\ppm$ (with $N=5$), when is used as a refinement of the seedless algorithm $\texttt{QPADMM}$, which solves \eqref{eq:exp_convex} via the alternating direction method of multipliers (ADMM). This setting has also been considered in the numerical experiments in \citep{deg_prof, Grampa}. We plot the average performance over $25$ Monte Carlo runs of the methods $\texttt{QPADMM}$ and $\texttt{QPADMM+PPMGM}$ (its refinement), and we include the performance of $\texttt{Grampa}$ and $\texttt{Grampa+PPMGM}$ for comparison. As before, the shaded area contains $90\%$ of the Monte Carlo runs. It is clear that $\texttt{QPADMM+PPMGM}$ outperforms the rest which is a consequence of the good quality of the seed of $\texttt{QPADMM}$. The caveat is that $\texttt{QPADMM}$ takes much longer to run than $\texttt{Grampa}$. In our experiments, for $n=200$, it is $2.5$ times slower on average, although in \citep{Grampa} a larger gap is reported for $n=1000$. This shows the scalability issues of $\texttt{QPADMM}$ which is not surprising considering that general purpose convex solvers are usually much slower than first order methods.

\begin{figure}[!ht]
    \centering
    \includegraphics[width=0.47\textwidth]{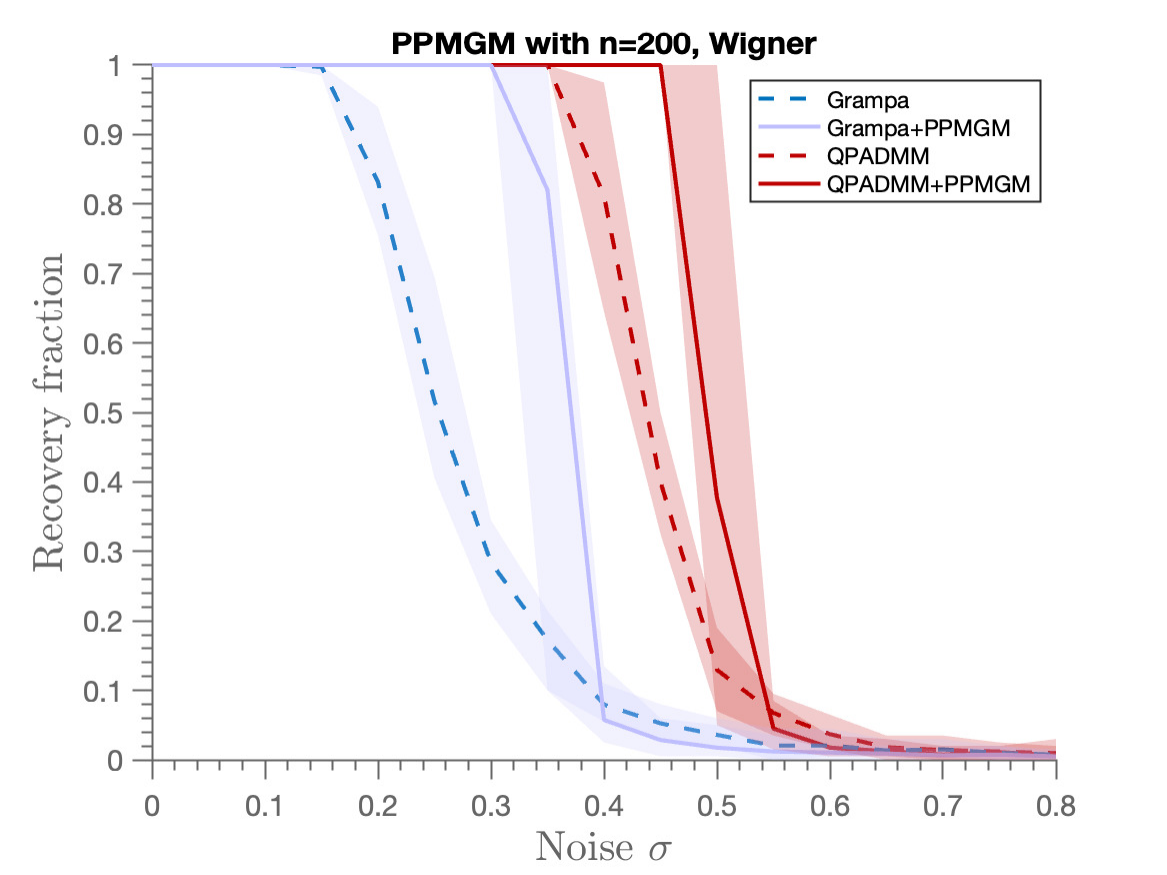}
    \caption{Performance of $\ppm$ (with $N=5$) used as a refinement (post-processing) of \texttt{QPADMM} for $n=200$ and $25$ Monte Carlo runs. We include the results of \texttt{Grampa} and its refinement for comparison purposes.}
    \label{fig-convex}
\end{figure}

\paragraph{Varying the number of iterations $N$.}
We experimentally evaluate the performance of \texttt{PPMGM} when varying the number of iterations $N$ in Algorithm \ref{alg:ppmgm}. In Figure \ref{fig:perf-2} we plot the recovery rate of \texttt{PPMGM}, initialized with $x^{(0)}$, with an overlap of $0.1$ with the ground truth. In Fig. \ref{fig2-a} we see that adding more iterations increases the performance of the algorithm for $n=500$; however the improvement is less pronounced in the higher noise regime. In other words, the number of iterations cannot make up for the fact that the initial seed is of poor quality (relative to the noise level). We use $N=1,2,4,8,30$ iterations and we observe a moderate gain between $N=8$ and $N=30$. In Fig. \ref{fig2-b} we use a matrix of size $n=1000$ and we see that the difference between using $N=1$ and $N>1$ is even less pronounced (we omit the case of $30$ iterations for readability purposes, as it is very similar to $N=8$). \revb{ This is in concordance with our main results, as the main quantities used by our algorithm are getting more concentrated as $n$ grows. In the case of one iteration, Proposition \ref{prop:diago_dom} says that the probability that the diagonal elements of the gradient term $AXB$ are the largest in their corresponding row is increasing with $n$ (we recall that we can assume that the ground truth is the identity w.l.o.g), which means that the probability of obtaining exact recovery, after the \texttt{GMWM} rounding, is increasing with $n$. This is verified experimentally here (comparing the blue curve in Fig. \ref{fig2-a} and Fig. \ref{fig2-b}). An analogous reasoning follows from Lemmas \ref{lem:nb_ngbh1_mt} and \ref{lem:nb_ngbh2_mt} in the case of multiple iterations. Ultimately, when $n$ increases, the relative performance of $\ppm$ with $N=1$ increases and there is less room for improvement using more iterations (altough the improvement is still significant).}
\begin{figure} 
  \begin{subfigure}[t]{0.475\textwidth}
    \includegraphics[width=\textwidth]{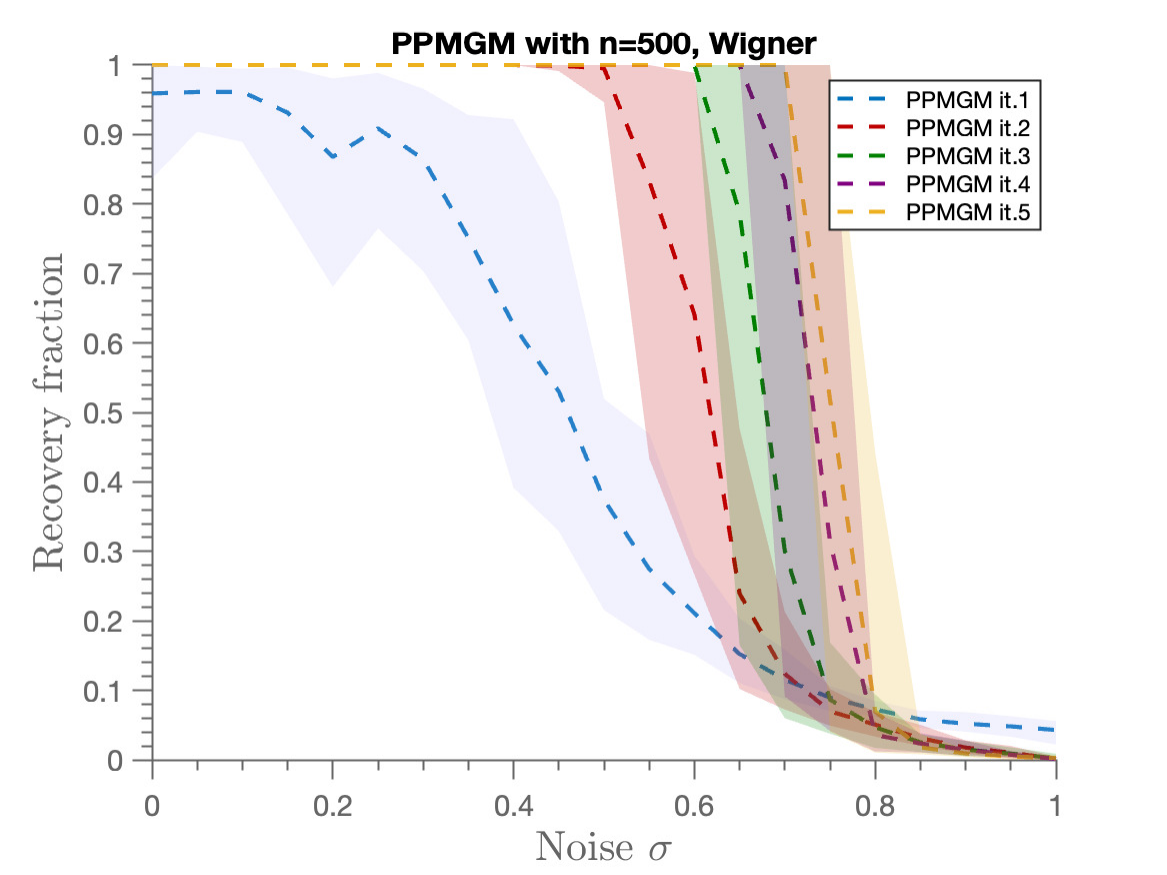}
    \caption{\texttt{PPMGM} with an initialization such that $\operatorname{overlap}(x^{(0)},x^*)=0.1$. Here $it.1,it.2,it.3,it.4,it.5$ corresponds to $1,2,4,8$ and $30$ iterations respectively. }
    \label{fig2-a}
  \end{subfigure}\hfill
  \begin{subfigure}[t]{0.475\textwidth}
    \includegraphics[width=\textwidth]{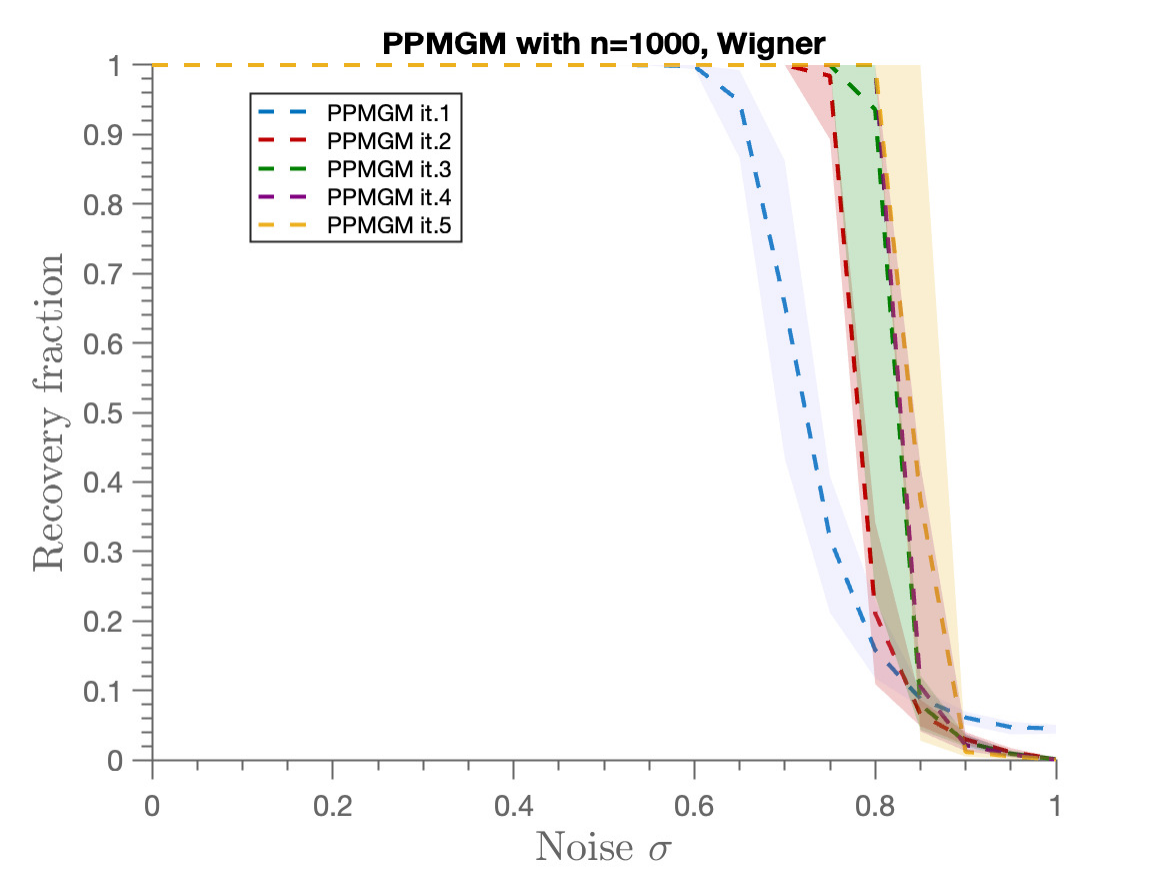}
    \caption{Here $it.1,it.2,it.3,it.4$ corresponds to $1,2,4$ and $8$ iterations respectively.}
    \label{fig2-b}
  \end{subfigure}
  \caption{Comparison of the performance of $\ppm$ with different values of $N$ (number of iterations).} 
  \label{fig:perf-2}
\end{figure}

\subsection{Sparsification strategies}\label{sec:spar_st}

Here we run \texttt{PPMGM} using different input matrices which are all transformations of the Wigner correlated matrices $A,B$. Specifically, we compare \texttt{PPMGM} (with $N=5$) with $A,B$ as input with the application of \texttt{PPMGM} to three different pairs of input matrices ($A^{\operatorname{spar}_1},B^{\operatorname{spar}_1}$), ($A^{\operatorname{spar}_2},B^{\operatorname{spar}_2}$) and ($A^{\operatorname{spar}_3},B^{\operatorname{spar}_3}$) that are defined as follows. 
\begin{align*}
    A^{\operatorname{spar}_1}_{ij}&=\mathbbm{1}_{|A_{ij}|<\tau};\enskip B^{\operatorname{spar}_1}_{ij}=\mathbbm{1}_{|B_{ij}|<\tau}, \\
    A^{\operatorname{spar}_2}_{ij}&=A_{ij}\mathbbm{1}_{|A_{ij}|<\tau};\enskip B^{\operatorname{spar}_2}_{ij}=B_{ij}\mathbbm{1}_{|B_{ij}|<\tau}, \\
    A^{\operatorname{spar}_3}_{ij}&=A_{ij}\mathbbm{1}_{|A_{ij}|\in \operatorname{top_k}(A_{i:})};\enskip B^{\operatorname{spar}_2}_{ij}=B_{ij}\mathbbm{1}_{|B_{ij}|\in \operatorname{top_k}(B_{i:})}, 
\end{align*}
where $\tau>0$ and for $k\in\mathbb{N}$ and a $n\times n$ matrix $M$,  $\operatorname{top_k}(M_{i:})$ is the set of the $k$ largest elements (breaking ties arbitrarily) of $M_{i:}$ (the $i$-th row of $M$). The choice of the parameter $\tau$ is mainly determined by the sparsity assumptions in \citep[Thm.B]{MaoRud}, \emph{i.e.}, if $G,H$ are two CER graphs to be matched with connection probability $p$ (which is equal to $qs$ in the definition \eqref{eq: ER_def}), then the assumption is that 
\begin{equation}\label{eq:sparsity_assump}
    (1+\epsilon)\frac{\log n}n\leq p\leq n^{\frac{1}{R\log\log n}-1}
\end{equation}
where $\epsilon>0$ is arbitrary and $R$ is an absolute constant. We refer the reader to \citep{MaoRud} for details. For each $p$ in the range defined by \eqref{eq:sparsity_assump} we solve the equation 
\begin{equation}\label{eq:param_tau}
    \prob(|A_{ij}|\leq \tau_p)=2\Phi(-\tau_p\sqrt n)=p
\end{equation}
where $\Phi$ is the standard Gaussian cdf (which is bijective so $\tau_p$ is well defined). In our experiments, we solve  \eqref{eq:param_tau} numerically. Notice that $A^{\operatorname{spar}_1}$ and $ B^{\operatorname{spar}_1}$ are sparse \rev{CER} graphs with a correlation that depends on $\sigma$. For the value of $k$ that defines $A^{\operatorname{spar}_3},B^{\operatorname{spar}_3}$ we choose $k=\Omega(\log n)$ or $k=\Omega(n^{o(1)})$, to maintain the sparsity degree in \eqref{eq:sparsity_assump}. In Figure \ref{fig:spar} we plot the performance comparison between the \texttt{PPMGM} without sparsification, and the different sparsification strategies. We see in Figs. \ref{figsp-a} and \ref{figsp-b} (initialized with overlap $0.5$ and $0.1$) that the use of the full information $A,B$ outperforms the sparser versions in the higher noise regimes  and for when the overlap of the initial permutation is small. On the other hand, the performance tends to be more similar for low levels of noise and moderately large number of correct initial seeds. In theory, sparsification strategies have a moderate denoising effect (and might considerably speed up computations), but this process seems to destroy important correlation information. 
\begin{figure}[!ht]
  \begin{subfigure}[t]{0.475\textwidth}
    \includegraphics[width=\textwidth]{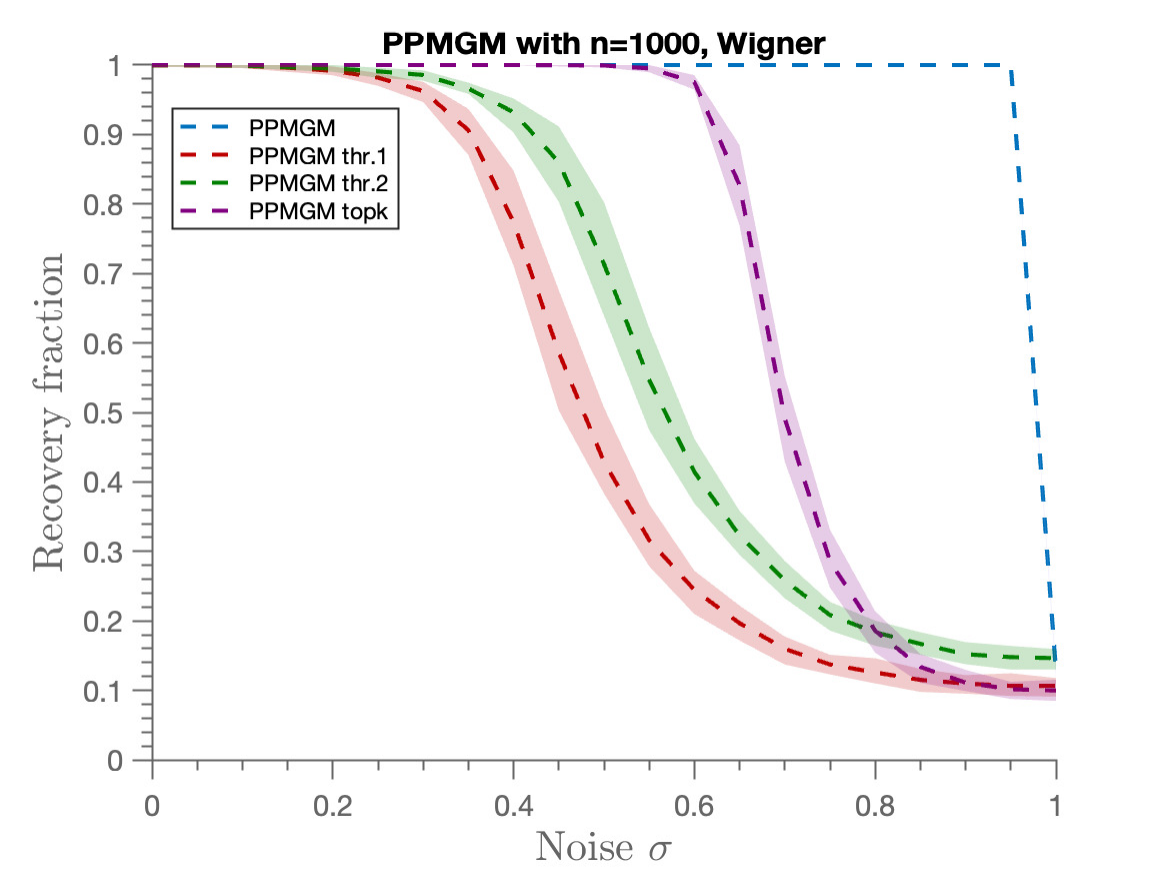}
    \caption{Initial overlap is equal to $0.5$}
    \label{figsp-a}
  \end{subfigure}\hfill
  \begin{subfigure}[t]{0.475\textwidth}
    \includegraphics[width=\textwidth]{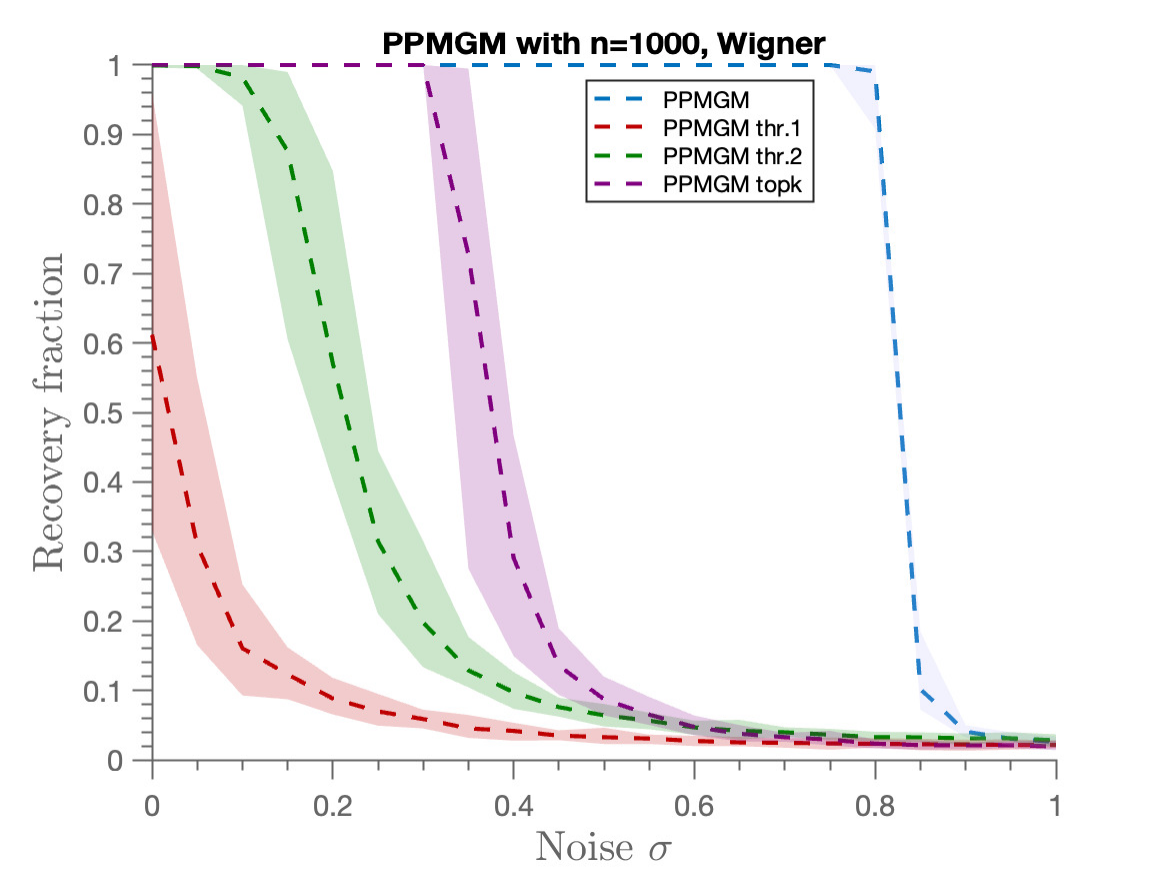}
    \caption{Initial overlap is equal to $0.1$}
    \label{figsp-b}
  \end{subfigure}
  \caption{Comparison between \texttt{PPMGM} (with $N=5$) with and without sparsification. Here $thr.1$ corresponds to the pair of matrices ($A^{\operatorname{spar}_1},B^{\operatorname{spar}_1}$), $thr.2$ corresponds to the pair  ($A^{\operatorname{spar}_2},B^{\operatorname{spar}_2}$) and top $k$ corresponds to ($A^{\operatorname{spar}_3},B^{\operatorname{spar}_3}$)} 
  \label{fig:spar}
\end{figure}

\subsubsection{Choice of the sparsification parameter $\tau$}\label{sec:tau_sel}
Solving \eqref{eq:param_tau} for $p$ in the range \eqref{eq:sparsity_assump} we obtain a range of possible values for the sparsification parameter $\tau$. To choose between them, we use a simple grid search where we evaluate the recovery rate for each sparsification parameter on graphs of size $n=1000$, and take the mean over $25$ independent Monte Carlo runs. In Fig. \ref{fig-hm}, we plot a heatmap with the results. We see that the best performing parameter in this experiment was for $\tau_5$ corresponding to a probability $p_5=51\times 10^{-3}$, although there is a moderate change between all the choices for $p$. 
\begin{figure}[!ht]
    \centering
    \includegraphics[width=0.57\textwidth]{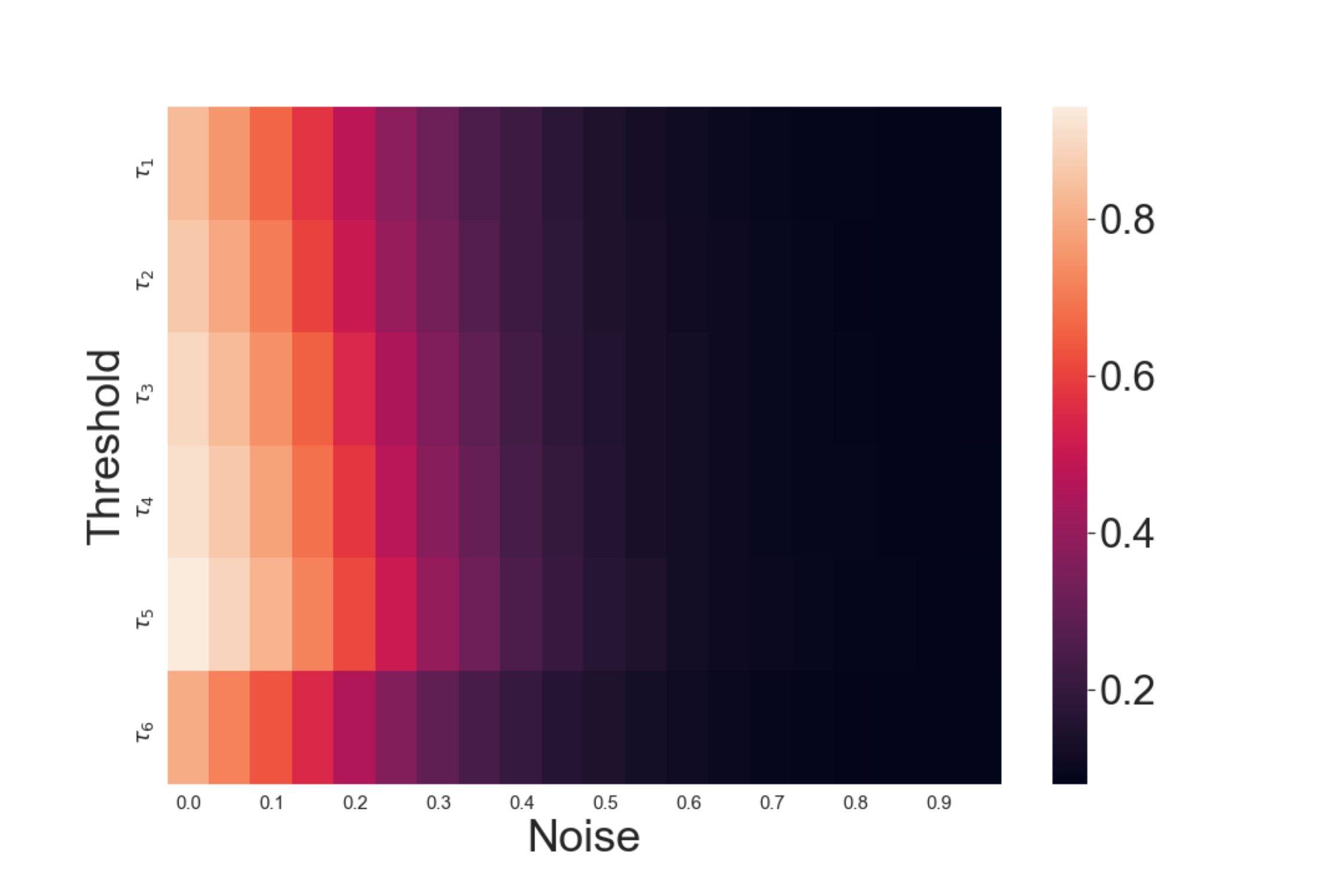}
    \caption{Heatmap for the recovery rate of \texttt{PPMGM} algorithm with input ($A^{\operatorname{spar}_1},B^{\operatorname{spar}_1}$) for different threshold values $\tau_i$($y$ axis); $i=1,\cdots,6$, and different values of $\sigma$ ($x$ axis). Here $\tau_i$ corresponds to the solution of \eqref{eq:param_tau} with $n=1000$ and $p_i$ for $i=1,2\cdots,6$ in a uniform grid between $p_1=42\times 10^{-3}$ and $p_6=54\times 10^{-3}$.}
    \label{fig-hm}
\end{figure}
\subsection{Real data}\label{sec:real_data}
We evaluate the performance of $\ppm$ for the task of matching 3D deformable objects, which is fundamental in the field of computer vision. We use the SHREC'16 dataset \citep{shrec} which contains $25$  shapes of kids (that can be regarded as perturbations of a single reference shape) in both high and low-resolution, together with the ground truth assignment between different pairs of shapes. Each image is represented by a triangulation (a triangulated mesh graph) which is converted to a weighted graph by standard image processing methods \citep{Peyre}. More specifically, each vertex corresponds to a point in the image (its triangulation) and the edge weights are given by the distance between the vertices. We use the low-resolution dataset in which each image is codified by a graph whose size varies from $8608$ to $11413$ vertices and where the average number of edges is around $0.05\%$ (high degree of sparsity). The main objective in this section is to show experimental evidence that $\ppm$ improves the quality of a matching given by a seedless algorithm, and for that, the pipeline is as follows.

\begin{enumerate}

    \item We first make the input graphs of the same size by erasing,  uniformly at random, the vertices of the larger graph.
    \item We run \texttt{Grampa} algorithm to obtain a matching.
    \item We use the output of \texttt{Grampa} as the initial point for $\ppm$.
    
\end{enumerate}

We present an example of the results in Figure \ref{fig-real} where we choose two images and find a matching between them following the above steps. We choose these two images based on the fact that the sizes of the graphs that represent them are the most similar in the whole dataset ($11265$ and $11267$ vertices). We can see visually  that the final matching obtained by $\ppm$  maps mostly similar parts of the body in the two shapes.

\begin{figure}[!ht]
    \centering
    \begin{subfigure}[t]{0.32\textwidth}
    \includegraphics[width=\textwidth]{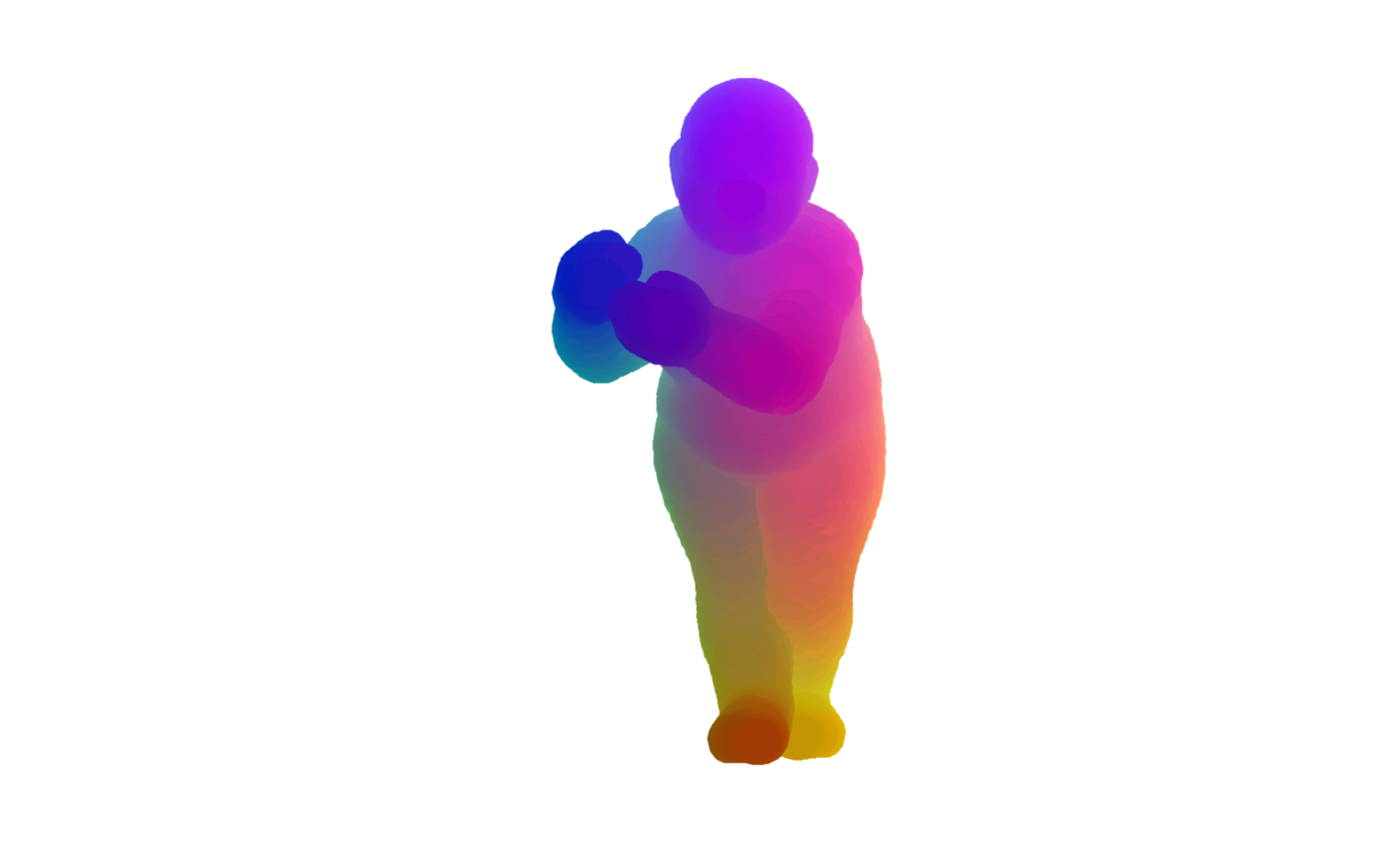}
    \caption{}
    \label{fig_real-a}
  \end{subfigure}%
  \begin{subfigure}[t]{0.32\textwidth}
    \includegraphics[width=\textwidth]{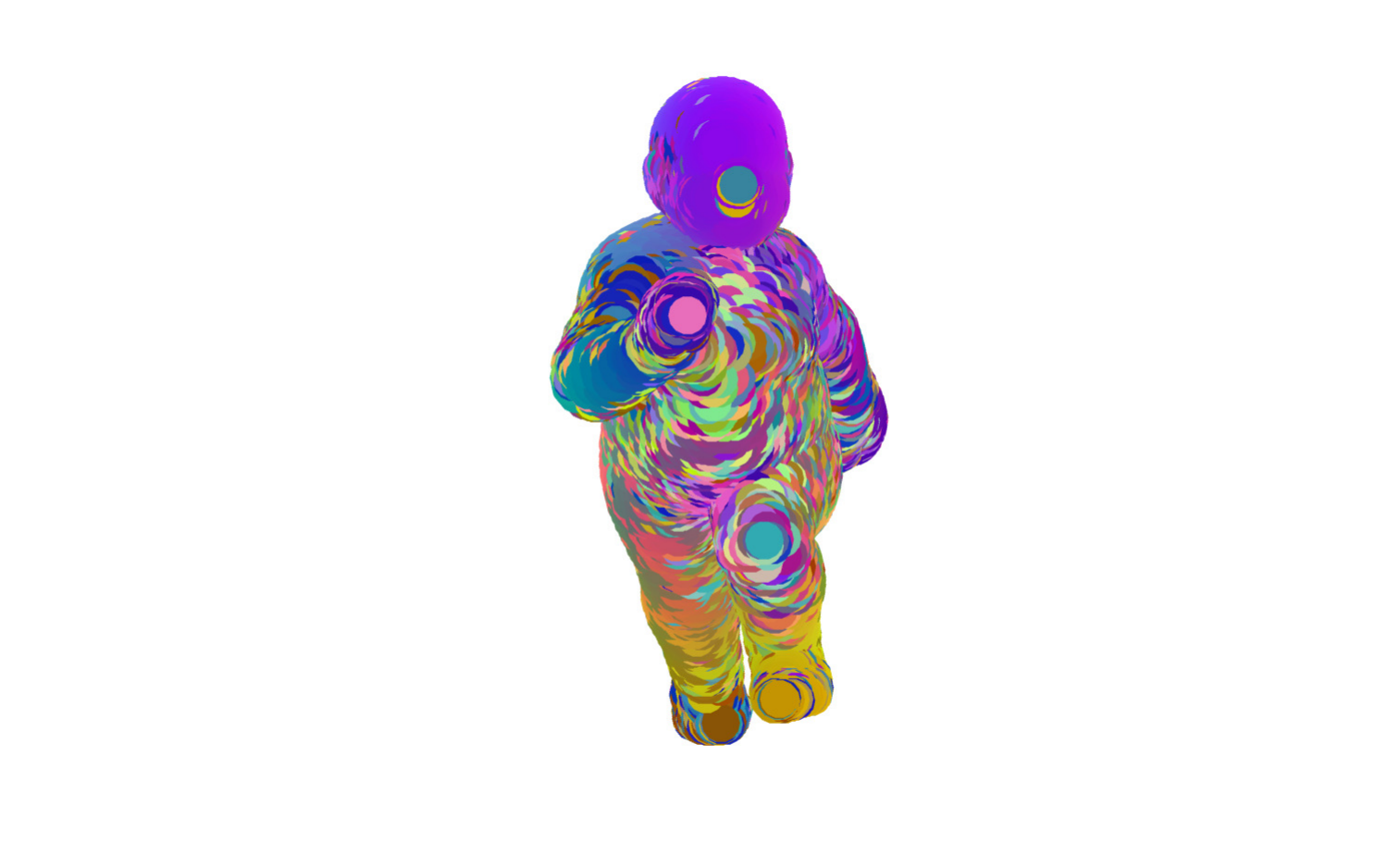}
    \caption{}
    \label{fig_real-b}
  \end{subfigure}%
  \begin{subfigure}[t]{0.32\textwidth}
    \includegraphics[width=\textwidth]{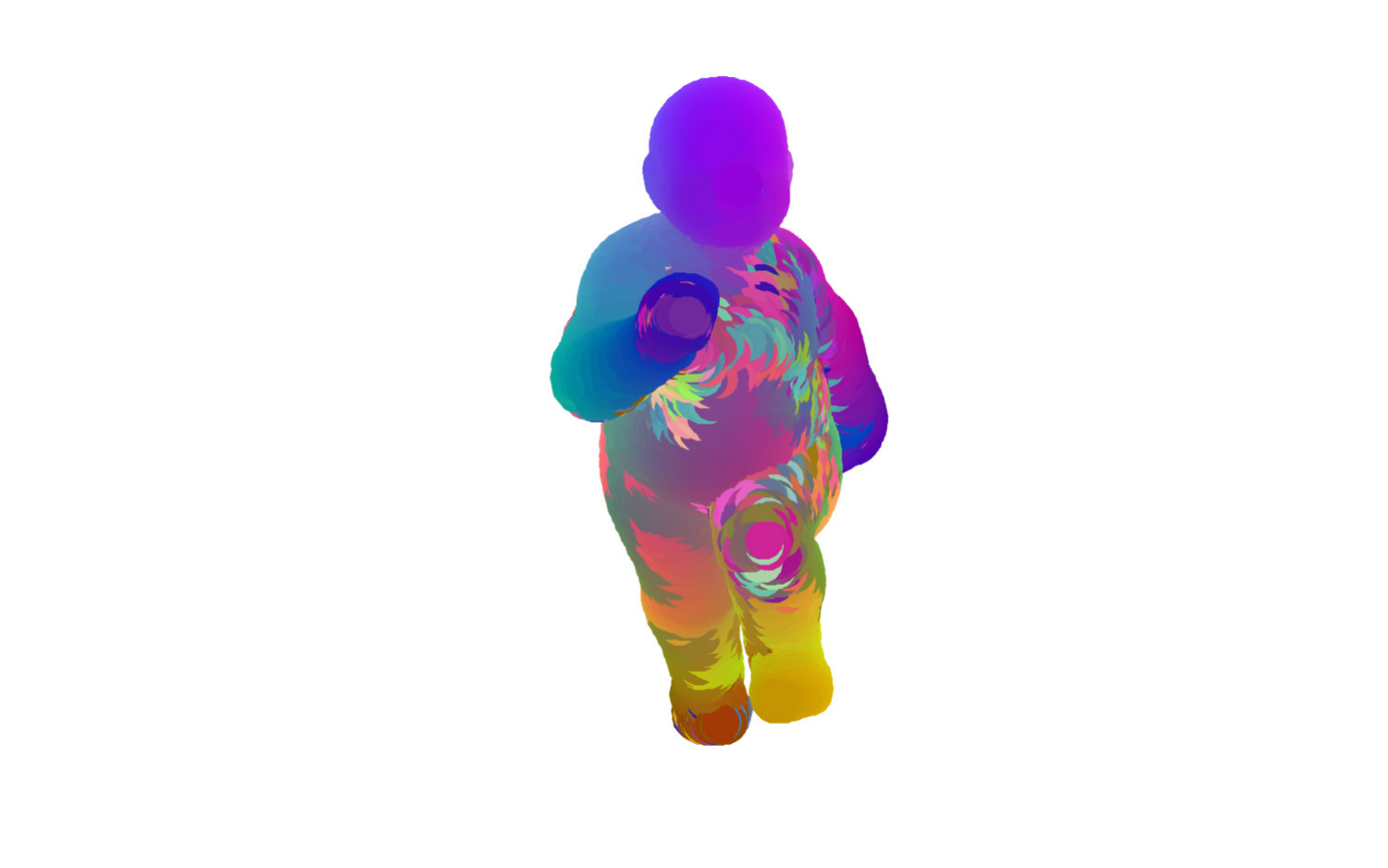}
    \caption{}
    \label{fig_real-c}
  \end{subfigure}
    \caption{Visual representation of the performance of $\ppm$ (with $100$ iterations), improving the results of \texttt{Grampa} algorithm. Fig. \ref{fig_real-a} is the reference shape (graph $A$)  where different colours are assigned to different parts of the body. Fig. \ref{fig_real-b} is the second shape (graph $B$) where the colours are assigned according to the matching given by \texttt{Grampa} algorithm; the accuracy is around $27\%$. Fig. \ref{fig_real-c} shows the results of $\ppm$, taking as the initial point the output of \texttt{Grampa} in Fig. \ref{fig_real-b}. The accuracy is improved to around $57\%$.}
    \label{fig-real}
\end{figure}
\revb{Following the experimental setting in \citep{YuXuLin}, we evaluate the performance of $\ppm$ using the Princeton benchmark protocol \citep{princeton_protocol}. Here we take all the pairs of shapes in the SHREC'16 dataset and apply the steps $1$ to $3$, of the pipeline described above, to them. Given graphs $A$ and $B$ (corresponding to two different shapes) we compute the normalized geodesic error as follows. For each node $i$ in the shape $A$, we compute $d_{B}(\hat{\pi}_{\operatorname{ppm}}(i),\pi^*(i))$, where $\hat{\pi}$ is the output of $\ppm$, $\pi^*$ is the ground truth matching between $A$ and $B$ and $d_{B}$ is the geodesic distance on $B$ (computed as the weighted shortest path distance using the triangulation representation of the image \citep{Peyre}). We then define the normalized error as $\varepsilon_{\operatorname{ppm}}(i):=d_{B}(\hat{\pi}_{\operatorname{ppm}}(i),\pi^*(i))/\sqrt{\operatorname{Area}(B)}$, where $\operatorname{Area}(B)$ is the surface area of $B$ (again computed using the triangulation representation). Then the cumulative distribution function (CDF) is defined as follows.\[\operatorname{CDF}_{\operatorname{ppm}}(\epsilon) = \sum^{n_A}_{i=1}\mathbbm{1}_{\varepsilon_{\operatorname{ppm}}(i)\leq \epsilon},\] where $n_A$ is the number of nodes of $A$.}

\revb{In Figure \ref{fig:CDF} we compare $\operatorname{CDF}_{\operatorname{ppm}}$  with $\operatorname{CDF}_{\operatorname{grampa}}$, which is defined in an analogous way by using $\varepsilon_{\operatorname{grampa}}(i):=d_{B}(\hat{\pi}_{\operatorname{grampa}}(i),\pi^*(i))/\sqrt{\operatorname{Area}(B)}$ instead of $\varepsilon_{\operatorname{ppm}}$ (here $\hat{\pi}_{\operatorname{grampa}}$ is output of $\texttt{Grampa}$). We observe that the performance increases overall for all the values of $\epsilon\in[0,1]$. In particular, the average percentage of nodes correctly matched (corresponding to $\epsilon=0$) increases from less than $15\%$ in the $\texttt{Grampa}$ baseline to more than $50\%$ with $\ppm$. Compared to the results in \citep{YuXuLin}, the performance of $\ppm$ is similar to their $1$-\texttt{hop} algorithm (although here we used the weighted adjacency matrix). The performance of $\ppm$ is slightly worse than what they reported for the $2$-\texttt{hop} algorithm. Regarding the latter, it is worth noting that, although the iterated application of their $2$-\texttt{hop} algorithm performs well in their experiments, no theoretical guarantees are provided beyond the case of one iteration. It is possible that our analysis can be extended to the case of multiple iterations of the $2-$\texttt{hop} algorithm, but this is beyond the scope of the present paper.}
\begin{figure}
    \centering    
    \includegraphics[scale=0.2]{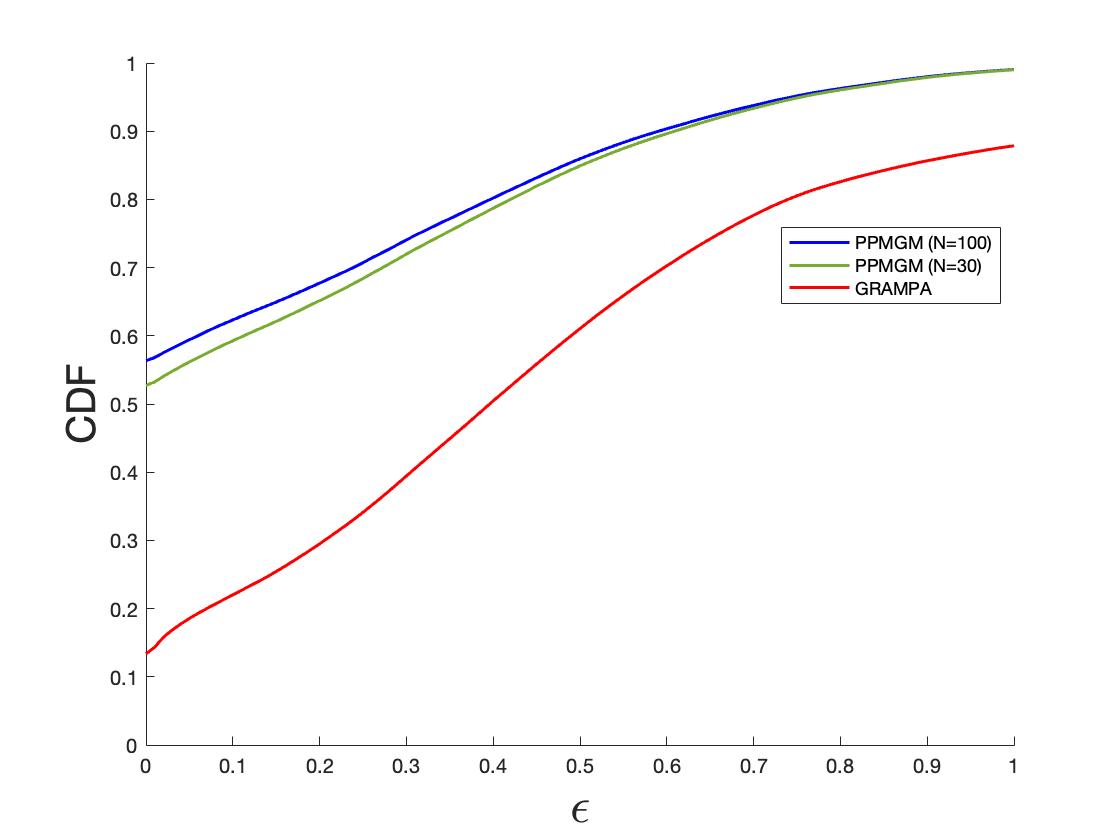}
    \caption{\revb{Average performance of $\ppm$ algorithm, with $N=30$ and $N=100$ iterations, when used to boost the performance of $\texttt{Grampa}$ algorithm. The average is over all pairs of graphs (shapes) of the SHREC'16 database.}}
    \label{fig:CDF}
\end{figure}

\section{Concluding remarks}
In this work, we analysed the performance of the projected power method (proposed in \citep{Villar}) as a seeded graph matching algorithm, in the correlated Wigner model. We proved that for a non-data dependent seed with $\mathcal{O}(\sqrt{n\log n})$ correctly pre-assigned vertices, the PPM exactly recovers the ground truth matching in one iteration. This is analogous to the state-of-the-art results for algorithms in the case of relatively sparse correlated \erdos-R\'enyi graphs. We additionally proved that the PPM can exactly recover the optimal matching in $\mathcal{O}(\log n)$ iterations for a seed that contains $\Omega\big((1-\kappa)n\big)$ correctly matched vertices, for a constant $\kappa\in (0,1)$, even if the seed can potentially be dependent on the data. For the latter result, we extended the arguments of \citep{MaoRud} from the (sparse) CER model to the (dense) CGW case, providing a uniform control on the error when the seed contains $\Omega\big((1-\kappa)n\big)$ fixed points. This provides theoretical guarantees for the use of PPM as a refinement algorithm (or a post-processing step) for other seedless graph matching methods. 

An open question is to find an efficient initialization method which outputs a permutation with order $(1-\kappa)n$ correctly matched vertices in regimes with higher $\sigma$ (say for $\sigma>1/2$). For those noise levels, spectral methods do not seem to perform well (at least in the experiments). An idea could be to adapt the results \citep{MaoRud} from the sparse CER case to the CGW case. In that paper, the authors construct for each vertex a signature containing the neighborhood information of that vertex and which is encoded as tree. Then a matching is constructed by matching those trees. It is however unclear how to adapt those results (which heavily rely on the sparsity) to the CGW setting.

\section*{Acknowledgements}
Most of the work was done while G.B. was a Ph.D. student at Inria Lille.

\bibliography{biblio}
\newpage
\appendix

\section{Proof of Proposition \ref{prop:diago_dom}}\label{app:concentration}

We divide the proof into two subsections. In Appendix \ref{app:diagdom_row_noiseless} we prove Lemma \ref{lem:tailbounds} and in  Appendix \ref{app:diagdom_row_noise} we prove part $(ii)$ of Proposition \ref{prop:diago_dom}. Before proceeding, let us introduce and recall some notation. Define $C':=AXA$ and $C'':=AXZ$, then  $C=AXB=\sqrt{1-\sigma^2}C'+\sigma C''$. Recall that for a permutation $x$, $S_X$ will denote the set of fixed points of $x$ (the set of non-zero diagonal terms of its matrix representation $X$) and we will often write $s_x=|S_X|/n=Tr(X)/n$. We will say that a real random variable $Y\sim\chi^2_K$ if it follows a central Chi-squared distribution with $K$ degrees of freedom.  

\subsection{Proof of Lemma \ref{lem:tailbounds}}\label{app:diagdom_row_noiseless}

The proof of Lemma \ref{lem:tailbounds} mainly revolves around the use of concentration inequalities for quadratic forms of Gaussian random vectors. For that, it will be useful to use the following representation of the entries of $C$. 
\begin{equation}\label{eq:Cdecom2}
C_{ij}=\langle A_{:i},XA_{:j}\rangle
\end{equation}
where we recall that $A_{:k}$ represents the $k$-th column of the matrix $A$. 
\begin{proof}[Proof of Lemma \ref{lem:tailbounds}]
~ \paragraph{High probability bound for $C_{ii}$.}
Define $\tilde{a}_i$ to be a vector in $\mathbb{R}^n$ such that 
\begin{equation*}
\tilde{a}_i(k)=\begin{cases}
    A_{ki},\text{ for } k\notin {i,x^{-1}(i)},\\
    \frac1{\sqrt{2}}A_{ii},\text{ for } k\in {i,x^{-1}(i)}.
    \end{cases}
\end{equation*} 
Using representation \eqref{eq:Cdecom2} we have \[C_{ii}=\langle \tilde{a}_i,X\tilde{a}_i\rangle + \calZ_i\]
 where
 \begin{equation*}
      \calZ_i:=\frac12A_{ii}\big(A_{x(i)i})+A_{x^{-1}(i)i}\big).
 \end{equation*}
It is easy to see that $\sqrt{n}\tilde{a}_i$ is a standard Gaussian vector. Using Lemma \ref{lem:dist_gaussian_inner} we obtain 
\[
n\langle\tilde{a}_i,X\tilde{a}_i\rangle\stackrel{d}{=} \sum^{n_1}_{i=1}\mu_ig^2_i-\sum^{n_2}_{i=1}\nu_ig'^2_i
\]
where $(\mu_i)^{n_1}_{i=1}, (-\nu_i)^{n_2}_{i=1}$, (with $\mu_i\geq 0$, $\nu_i\geq0$ and $n_1+n_2=n$) is the sequence of eigenvalues of $\frac12(X+X^T)$ and $g=(g_1,\cdots, g_{n_1})$, $g'=(g'_1,\cdots, g'_{n_2})$ are two independent sets of i.i.d standard Gaussians. Lemma \ref{lem:dist_gaussian_inner} tell us in addition that $\|\mu\|_1-\|\nu\|_1=s_xn$, $\|\mu\|_2+\|\nu\|_2\leq \sqrt{2n}$ and $\|\mu\|_\infty,\|\nu\|_\infty\leq 1$. Using Corollary \ref{cor:lau_mass}  \eqref{eq:lau_mass_lt}, we obtain 
\begin{equation}\label{eq:bound_aXa}
    \prob(n\langle\tilde{a}_i,X\tilde{a}_i\rangle\leq s_xn-2\sqrt{2nt}-2t)\leq e^{-t}
\end{equation}
for all $t\geq 0$. To obtain a concentration bound for $\calZ_i$ we will distinguish two cases. 

\noindent \textit{(a)\underline{Case $i\in S_X$.}} In this case, we have $\calZ_i=a^2_i(i)$, which implies that $C_{ii}\geq \langle\tilde{a}_i,X\tilde{a}_i\rangle$. Hence \[\prob(nC_{ii}\leq s_xn-2\sqrt{2nt}-2t)\leq 2e^{-t}.\]
Replacing $t=\overline{t}:=\frac{n}{2}(\sqrt{1+\frac{s_x}2}-1)^2$ in the previous expression, one can verify\footnote{Indeed, the inequality $(\sqrt{1+x}-1)^2\geq \frac16x^2$, follows from the inequality $x^2+(2\sqrt{6}-6)x\leq 0$, which holds for $0<x\leq 1$.} that $\overline{t}\geq \frac{n}{48}s_x^2$, for $s_x\in (0,1]$, hence 
\begin{equation*}
    \prob(C_{ii}\leq s_x/2)\leq 2e^{-\frac{s_x^2}{48}n}
\end{equation*}
which proves \eqref{eq:bounddiag} in this case. 

\noindent \textit{(b) \underline{Case $i\notin S_X$.}}  Notice that in this case, $a_i(i)$ is independent from $(a_i(x(i))+a_i(x^{-1}(i))$, hence $n\calZ_i\stackrel{d}{=}g_1g_2$, where $g_1,g_2$ are independent standard Gaussians. Using the polarization identity $g_1g_2=\frac14(g_1+g_2)^2-\frac14(g_1-g_2)^2$, we obtain  \[n\calZ_i\stackrel{d}{=}\frac12(\tilde{g}^2_1-\tilde{g}^2_2)\]
where $\tilde{g_1},\tilde{g_2}$ are independent standard Gaussians. By Corollary \ref{cor:lau_mass} we have 
\begin{equation}\label{eq:boundZ_i}
    \prob\Big(2n\calZ_i\leq -4\sqrt{t}-2t\Big)\leq 2e^{-t}.
\end{equation}
Using \eqref{eq:bound_aXa} and \eqref{eq:boundZ_i}, we get 
\begin{equation*}
    \prob(nC_{ii}\leq s_xn-2(\sqrt {2n}+1)\sqrt{t}-3 t)\leq 4e^{-t}
\end{equation*}
 or, equivalently 
\begin{equation}\label{eq:bound_Cii}
    \prob\left(C_{ii}\leq s_x-2(\sqrt2+1/\sqrt{n})\sqrt{\frac tn}-3 \frac tn\right)\leq 4e^{-t}.
\end{equation}
Replacing $t=\overline{t}:=\frac{n}{36}\big(\sqrt{d^2+6s_x}-d\big)^2$, where $d=2(\sqrt2+1/\sqrt{n})$, in the previous expression and noticing that $\overline{t}\geq \frac{1}{6}s_x^2n$, we obtain the bound
\begin{equation*}
    \prob(C_{ii}\leq s_x/2)\leq 4e^{-\frac{s^2_x}{6}n}.
\end{equation*}

\paragraph{High probability bound for $C_{ij}$, $i\neq j$.} Let us first define the vectors $\tilde{a}_i,\tilde{a}_j\in \mathbb{R}^{n}$ as \begin{equation*}
\tilde{a}_i(k):=\begin{cases}
A_{ki},\enskip \text{ for }k\notin\{j,x^{-1}(i)\},\\
0,\enskip \text{ for }k\in\{j,x^{-1}(i)\},\\
\end{cases}
\end{equation*} and 
\begin{equation*}
\tilde{a}_j(k):=\begin{cases}
A_{kj},\enskip \text{ for }k\notin\{j,x^{-1}(i)\},\\
0,\enskip \text{ for }k\in\{j,x^{-1}(i)\}.\\
\end{cases}
\end{equation*} Contrary to $a_i$ and $a_j$ which share a coordinate, the vectors $\tilde{a}_i$ and $\tilde{a}_j$ are independent. With this notation, we have the following decomposition \begin{equation*}
    C_{ij}=\langle \tilde{a}_i,X\tilde{a}_j\rangle+ A_{ji}\Big(A_{x(j)j}+A_{x^{-1}(i)i}\Big).
\end{equation*} 
For the first term, we will use the following polarization identity \begin{equation}\label{eq:polarization}
\langle \tilde{a}_i,X\tilde{a}_j\rangle= \|\frac12(\tilde{a}_i+X\tilde{a}_j)\|^2-\|\frac12(\tilde{a}_i-X\tilde{a}_j)\|^2.
\end{equation}
By the independence of $\tilde{a}_i$ and $\tilde{a}_j$, it is easy to see that $\tilde{a}_i+X\tilde{a}_j$ and $\tilde{a}_i-X\tilde{a}_j$ are independent Gaussian vectors and $\expec[\langle \tilde{a}_i,X\tilde{a}_j\rangle]=0$. Using \eqref{eq:polarization} and defining $\calZ_{ij}:=A_{ji}\Big(A_{x(j)j}+A_{x^{-1}(i)i}\Big)n$, it is easy to see that
\begin{equation}\label{eq:decomp_offdiag}
nC_{ij}\stackrel{d}{=}\sum^{n-1}_{i=1}\mu_ig^2_i-\sum^{n-1}_{i=1}\nu_ig'^2_i+\calZ_{ij}
\end{equation}
where $g_1,\cdots,g_{n-1}$ and $g'_1,\cdots,g'_{n-1}$ are two sets of independent standard Gaussian variables and $\mu_i,\nu_i\in \{\frac12,\frac34,1\}$, for $i\in[n-1]$. The sequences $(\mu_i)^{n-1}_{i=1},(\nu_i)^{n-1}_{i=1}$ will be characterised below, when we divide the analysis into two cases $x(j)=i$ and $x(j)\neq i$. We first state the following claim about $\calZ_{ij}$.
\begin{claim}\label{claim:distrZ} For $i\neq j$, we have 
\[\calZ_{ij}\stackrel{d}{=}\begin{cases}
q_{ij}(\zeta_1-\zeta_2)\enskip \text{ if }x(j)\neq i,\\
2\zeta_3 \enskip \text{ if }x(j)=i,
\end{cases}\]
where $\zeta_1,\zeta_2$ and $\zeta_3$ are independent Chi-squared random variables with one degree of freedom and 
\[q_{ij}=\begin{cases}
\sqrt\frac32 \enskip \text{ if }i\in S_X,j\notin S_X\text{ or }i\notin S_X,j\in S_X,\\
\sqrt 2 \text{ if }i,j\in S_X,\\
\frac1{\sqrt2 } \text{ if }i,j\notin S_X.
\end{cases}\]
\end{claim}
We delay the proof of this claim until the end of this section.
From the expression \eqref{eq:decomp_offdiag}, we deduce that the vectors $g=(g_1,\cdots,g_{n-1})$, $g'=(g'_1,\cdots,g'_{n-1})$ and $\calZ_{ij}$ are independent. Hence, by Claim \ref{claim:distrZ} the following decomposition holds 
\begin{equation*}
nC_{ij}\stackrel{d}{=}\sum^{n}_{i=1}\mu_ig^2_i-\sum^{n}_{i=1}\nu_ig'^2_i
\end{equation*}
where \[\mu_{n}=
\begin{cases}
q_{ij}\text{ if }x(j)\neq i, \\
2\text{ if }x(j)=i, 
\end{cases} \text{ and } 
\nu_{n}=
\begin{cases}
q_{ij}\text{ if }x(j)\neq i, \\
0\text{ if }x(j)=i.
\end{cases}
\]
Let us define $\mu:=(\mu_1,\cdots,\mu_{n})$ and $\nu:=(\nu_1,\cdots,\nu_{n})$. We will now distinguish two cases. 

\noindent \textit{(a) \underline{Case $x(j)\neq i$.}} In this case, we can verify that one of the $\mu_1,\cdots,\mu_{n-1}$ is equal to $0$ (and the same is true for the values $\nu_1,\cdots,\nu_{n-1}$). Assume without loss of generality that $\mu_1=\nu_1=0$. Also, one of the following situations must happen for the sequence $\mu_2,\cdots,\mu_{n-1}$ (resp. $\nu_2,\cdots,\nu_{n-1}$): either $n-3$ of the elements of the sequence are equal to $\frac12$ and one is equal $1$ or $n-4$ are equal to $\frac12$ and two are equal to $\frac34$ or $n-3$ are equal to $\frac12$ and one is equal to $\frac34$. In either of those cases, the following is verified     \begin{align*}
    \|\mu\|_1-\|\nu\|_1&=0,\\
    \|\mu\|_2+\|\nu\|_2&\leq \sqrt{2n},\\
    \|\mu\|_\infty,\|\nu\|_\infty&\leq \sqrt{2},
\end{align*} 
where the first equality comes from Lemma \ref{lem:expectation}, the inequality on the norm  $\|\cdot\|_2$ comes from the fact that in the worst case $\|\mu\|_2=\|\nu\|_2\leq \sqrt{\frac{n+1}4}$. The statement about the norm $\|\cdot\|_\infty$ can be easily seen by the definition of $\mu$ and $\nu$. Using \eqref{eq:lau_mass_ut}, we obtain 
\begin{equation*}
    \prob(nC_{ij}\geq 4\sqrt{nt}+4t)\leq 2e^{-t}.
\end{equation*}
Replacing $t=\overline{t}:=\frac{n}{4}(\sqrt{1+\frac{s_x}2}-1)^2$ in the previous expression and noticing that $\overline{t}\geq \frac1{96}s_x^2n$ for $s_x\in (0,1]$ leads to the bound 
\begin{equation*}
    \prob(C_{ij}\geq s_x/2)\leq 2e^{-\frac{s_x^2}{96}n}.
\end{equation*}

\noindent \textit{(b) \underline{Case $x(j)= i$.}} In this case, we have that for the sequence $\mu_1,\cdots,\mu_{n-1}$ (resp. $\nu_1,\cdots,\nu_{n-1}$): either $n-2$ of the elements of the sequence are equal to $\frac12$ and one is equal $1$ or $n-3$ are equal to $\frac12$ and two are equal to $\frac34$. In either case, the following holds
\begin{align*}
    \|\mu\|_1-\|\nu\|_1&=2,\\
    \|\mu\|_2+\|\nu\|_2&\leq 2\sqrt{n},\\
    \|\mu\|_\infty,\|\nu\|_\infty&\leq 2.
\end{align*}
Here, the inequalities for the norms $\|\cdot\|_1,\|\cdot\|_{\infty}$ follow directly from the definition of $\mu$ and $\nu$, and the inequality for $\|\cdot\|_2$ follows by the fact that, in the worst case, $\|\mu\|_2+\|\nu\|_2=\sqrt{\frac{n+6}4}+\sqrt{\frac{n+2}4}$.
Using \eqref{eq:lau_mass_ut}, we get 
\begin{equation*}
    \prob(nC_{ij}\geq 2+4\sqrt{nt}+4t)\leq 2e^{-t}.
\end{equation*}
Replacing $t=\overline{t}:=\frac{n}{4}(\sqrt{1+\frac{s_x}2-\frac2n}-1)^2$ in the previous expression and noticing that $\overline{t}\geq \frac1{20}s_x^2n$ for $s_x\in (10/n,1]$ we get 
\begin{equation*}
    \prob(C_{ij}\geq s_x/2)\leq 2e^{-\frac{s_x^2}{20}+4/n}\leq 3e^{-\frac{s_x^2}{20}},
\end{equation*}
where we used that $n\geq 10$.

\end{proof}
\begin{proof}[Proof of Claim \ref{claim:distrZ}]
Observe that when $x(j)=i$ (or equivalently $x^{-1}(i)=j$) we have $\calZ_{ij}=2nA^2_{ij}$. Given that $i\neq j$ by assumption, it holds $A^2_{ij}\sim \calN (0,\frac1n)$, which implies that $\calZ_{ij}\stackrel{d}{=}2\zeta_3$ for $\zeta_3\sim \chi^2_1$. In the case $x(j)\neq i$, let us define \begin{equation*}
    \psi_1:=\sqrt{n}A_{ij},\enskip\psi_2:=\sqrt{n}A_{jx(j)},\enskip\psi_3:=\sqrt{n}A_{ix^{-1}(i)},
\end{equation*} which are all independent Gaussians random variables. Moreover, $\psi_1\sim \calN(0,1)$ and 
\[\psi_2+\psi_3 \sim 
\begin{cases} \calN(0,2)\text{ if }i,j\notin S_X,\\
\calN(0,3) \text{ if }i\in S_X,j\notin S_X\text{ or }i\notin S_X,j\in S_X,\\
\calN(0,4)\text { if } i,j\in S_X.
\end{cases}\]
Consider the case $i,j\notin S_X$. In this case, it holds 
\begin{align*}
    \calZ_{ij}=\sqrt{2}\psi_1\Big(\frac{\psi_2+\psi_3}{\sqrt{2}}\Big)
    =\frac{1}{\sqrt{2}}{\Big(\frac{\psi_1}{\sqrt 2}+\frac{\psi_2+\psi_3}{2}\Big)}^2-\frac{1}{\sqrt{2}}{\Big(\frac{\psi_1}{\sqrt 2}-\frac{\psi_2+\psi_3}{2}\Big)}^2.
\end{align*}
Notice that $\frac{\psi_1}{\sqrt 2}+\frac{\psi_2+\psi_3}{2}$ and $\frac{\psi_1}{\sqrt 2}-\frac{\psi_2+\psi_3}{2}$ are independent standard normal random variables, hence $\calZ_{ij}\stackrel{d}{=}\frac1{\sqrt{2}}(\zeta_1-\zeta_2)$, where $\zeta_1$ and $\zeta_2$ are independent $\chi^2_1$ random variables. The proof for the other cases is analogous. 
\end{proof}

\subsection{Proof of Proposition \ref{prop:diago_dom} part $(ii)$}\label{app:diagdom_row_noise} 
%
%
Now we consider the case where $\sigma\neq 0$. It is easy to see that here the analysis of the noiseless case still applies (up to re-scaling by $\sqrt{1-\sigma^2}$) for the matrix $C'=AXA$. 
We can proceed in an analogous way for the matrix $C''=AXZ$ which will complete the analysis (recalling that $C=\sqrt{1-\sigma^2}C'+\sigma C''$). 

Before we proceed with the proof, we explain how the tail analysis of entries of $C'$ in Prop.\ref{prop:diago_dom} part $(i)$ helps us with the tail analysis of $C''$. Observe that for each $i,j\in[n]$ we have 
\begin{equation*}
   C''_{ij}=\sum_{k,k'}A_{ik}X_{k,k'}Z_{k',j}=\sum^n_{k=1}A_{ik}Z_{x(k)j}= \langle A_{:i},XZ_{:j}\rangle.
\end{equation*}
The term $C''_{ij}$, for all $i,j\in[n]$, can be controlled similarly to the term $C'_{i'j'}$ (when $i'\neq j'$). Indeed, we have the following 
\begin{lemma}\label{lem:tailbound_noise}
For $t\geq 0$ we have 
\begin{equation*}
    \prob(C''_{ij}\leq -4\sqrt{nt}-2t\big)=\prob(C''_{ij}\geq 4\sqrt{nt}+2t\big)\leq 2e^{-t}.
\end{equation*}
Consequently, 
\begin{equation*}
    \prob(C''_{ij}\geq s_x/2)\leq 2e^{-\frac{s_x^2}{96}n}.
\end{equation*}
\end{lemma}
\begin{proof}
We define $h_1:=\frac12(A_{:i}+XZ_{:j})$ and $h_2:=\frac12(A_{:i}-XZ_{:j})$. It is easy to see that $h_1$ and $h_2$ are two i.i.d Gaussian vectors of dimension $n$. By the polarization identity, we have 
\begin{align*}
   n \langle A_{:i},XZ_{:j}\rangle=n(\|h_1\|^2-\|h_2\|^2)
    \stackrel{d}{=}\sum^n_{i=1}\mu_ig^2_i-\sum^n_{i=1}\nu_ig'^2_i
\end{align*}
where $g=(g_1,\cdots,g_n)$  and $g'=(g'_1,\cdots,g'_n)$  are independent standard Gaussian vectors and the vectors $\mu=(\mu_1,\cdots,\mu_n),\nu=(\nu_1,\cdots,\nu_n)$ have positive entries that satisfy, for all $i\in [n]$, $\mu_i,\nu_i\in \{\frac1{\sqrt2},\sqrt{\frac34},1\}$. For $\mu_i$ (and the same is true for $\nu_i$) the following two cases can happen: either $n-1$ of its entries are $1/\sqrt{2}$ and one entry takes the value $1$ (when $i=j$) or $n-2$ of its entries are $1/\sqrt{2}$ and two entries take the value $\sqrt{3/4}$ (when $i\neq j$). In any of those cases, one can readily see that 
\[\|\mu\|_1=\|\nu\|_1,\enskip \|\mu\|_2+\|\nu\|_2\leq \sqrt{n},\enskip \|\mu\|_\infty,\|\nu\|_\infty\leq 1.\]
Using Corollary \ref{cor:lau_mass} we obtain 
\begin{align*}
    \prob\big(n(\|h_1\|^2-\|h_2\|^2)\geq 4\sqrt{nt}+2t\big)&\leq 2e^{-t},\\
    \prob\big(n(\|h_1\|^2-\|h_2\|^2)\leq -4\sqrt{nt}-2t\big)&\leq 2e^{-t}.
\end{align*}
Arguing as in the proof of Proposition \ref{prop:diago_dom} part $(i)$ we obtain the bound
\begin{equation*}
    \prob(C''_{ij}\geq s_x/2)\leq 2e^{-\frac{s_x^2}{96}n}.
\end{equation*}
\end{proof}
Now we introduce some definitions that will be used in the proof. We define $s_{\sigma,x}:=\frac12\sqrt{1-\sigma^2}s_x$, and for $\delta>0$, $i,j\in[n]$, we define the following events \[\mathcal{E}^i_\delta:=\{\sqrt{1-\sigma^2}C'_{ii}\leq s_{\sigma,x}+\delta \}\cup \{\sigma C''_{ii}\leq -\delta\},\]
\[\mathcal{E}^{ij}:=\{\sqrt{1-\sigma^2}C'_{ij}\geq s_{\sigma,x}/2 \}\cup \{\sigma C''_{ij}\geq s_{\sigma,x}/2 \}\enskip\text{, for }i\neq j.\]
One can easily verify that $\{C_{ii}\leq s_{\sigma,x}\}\subset \mathcal{E}^i_{\delta}$, hence it suffices to control the probability of $ \mathcal{E}^i_{\delta}$. For that we use the union bound and the already established bounds in Lemmas \ref{lem:tailbounds} and \ref{lem:tailbound_noise}. To attack the off-diagonal case, we observe that the following holds $\{C_{ij}\geq s_{\sigma,x}\}\subset \mathcal{E}^{ij}$. The following lemma allows us to bound the probability of the events $\mathcal{E}^i_\delta$ and $\mathcal{E}^{ij}$.
\begin{lemma}\label{lem:probaevents}
Let $\delta$ be such that $0\leq \delta\leq \frac{s_x}2\sqrt{1-\sigma^2}$.  
Then for $i,j\in[n]$ with $i\neq j$ have the following bounds 
\begin{align}\label{eq:eventdiag}
    \prob(\mathcal{E}^i_\delta)&\leq 4e^{-\frac{1}{96}(\frac{s_x}2-\frac\delta{\sqrt{1-\sigma^2}})^2n}+2e^{-\frac{1}{96}(\frac{\delta}{\sigma})^2n} \\ \label{eq:eventoffdiag}
    \prob(\mathcal{E}^{ij})& \leq 4e^{-\frac1{384} s_x^2(\frac{1-\sigma^2}{\sigma^2}\wedge 1)n}.
\end{align}
 In particular, we have
\begin{equation}\label{eq:eventdiag2}
\prob(\mathcal{E}^i_{\delta_{\sigma,x}})\leq 6e^{-\frac1{384} s_x^2(\frac{1-\sigma^2}{1+2\sigma\sqrt{1-\sigma^2}})n}
\end{equation}
where $\delta_{\sigma,x}=\frac{\sigma\sqrt{1-\sigma^2}}{\sigma+\sqrt{1-\sigma^2}}\frac{s_x}2$.
\end{lemma}
\begin{proof}
Using \eqref{eq:bound_Cii}, we have that 
\begin{equation*}
 \prob\Big(\sqrt{1-\sigma^2}C_{ii}\leq \sqrt{1-\sigma^2}\big(s_x-2(\sqrt2+1/\sqrt{n})\sqrt{\frac tn}-3 \frac tn\big)\Big)\leq 4e^{-t}.
\end{equation*}
Replacing $t=\overline{t}:=\frac{n}{36}\big(\sqrt{d^2+{6s_x-\frac{12\delta}{\sqrt{1-\sigma^2}}}}-d\big)^2$ in the previous expression, where $d=2(\sqrt2+1/\sqrt{n})$, and observing that 
$\overline{t}\geq \frac{1}{6}(\frac{s_x}{2}-\frac{\delta}{\sqrt{1-\sigma^2}})^2$, which is valid for $0\leq \delta\leq \frac{s_x}2\sqrt{1-\sigma^2}$, we obtain 
\begin{equation*}
    \prob\Big(\sqrt{1-\sigma^2}C'_{ii}\leq s_{\sigma,x}+\delta \Big)\leq 4e^{-\frac{1}{6}(\frac{s_x}{2}-\frac{\delta}{\sqrt{1-\sigma^2}})^2n}.
\end{equation*}
Using this and Lemma \ref{lem:tailbound_noise} we have
\begin{align}
    \prob(\mathcal{E}^i_\delta)&\leq \prob(\sqrt{1-\sigma^2}C'_{ii}\leq s_{\sigma,x}+\delta)+\prob(\sigma C''_{ii}\leq -\delta)\nonumber\\ 
    &\leq 4e^{-\frac{1}{6}(\frac{s_x}2-\frac\delta{\sqrt{1-\sigma^2}})^2n}+2e^{-\frac{1}{96}(\frac{\delta}{\sigma})^2n}.\nonumber
\end{align}
Similarly, to prove \eqref{eq:eventoffdiag} we verify that
\begin{align*}
    \prob(\mathcal{E}^{ij})&\leq \prob(C'_{ij}\geq \frac{s_x}4)+\prob(C''_{ij}\geq \frac{\sqrt{1-\sigma^2}}{\sigma}\frac{s_x}4)\\
    &\leq 2e^{-\frac{1}{384}s_x^2n}+2e^{-\frac1{384} s_x^2(\frac{1-\sigma^2}{\sigma^2})n}\\
    &\leq 4e^{-\frac1{384} s_x^2(\frac{1-\sigma^2}{\sigma^2}\wedge 1)n}. 
\end{align*}
To prove \eqref{eq:eventdiag2} it suffices to use  \eqref{eq:eventdiag} with the choice of $\delta=\delta_{\sigma,x}=\frac{\sigma\sqrt{1-\sigma^2}}{\sigma+\sqrt{1-\sigma^2}}\frac{s_x}2$.
\end{proof}
With this we prove the diagonal dominance for each fixed row of $C$.

\begin{proof}[Proof of Prop. \ref{prop:diago_dom} part $(ii)$]
Define $\tilde{\mathcal{E}}_j:=\{C_{ii}\leq s_{\sigma,x}\}\cup\{C_{ij}\geq s_{\sigma,x}\}$, which clearly satisfies $\{C_{ii}\leq C_{ij}\}\subset\tilde{\mathcal{E}}_j$. Then by the union bound, 
\begin{align*}
    \prob(\cup_{j\neq i}\tilde{\mathcal{E}}_j)&\leq \prob(C_{ii}\leq s_{\sigma,x})+\sum_{j\neq i} \prob(C_{ij}\geq s_{\sigma,x})\\
    &\leq \prob(\mathcal{E}^i_{\delta_{\sigma,x}})+\sum_{j\neq i}\prob(\mathcal{E}^{ij})\\
   &\leq 6e^{-\frac1{384} s_x^2(\frac{1-\sigma^2}{1+2\sigma\sqrt{1-\sigma^2}})n}+4(n-1)e^{-\frac1{384} s_x^2(\frac{1-\sigma^2}{\sigma^2}\wedge 1)n}\\
   &\leq 5ne^{-\frac1{384} s_x^2(\frac{1-\sigma^2}{1+2\sigma\sqrt{1-\sigma^2}})n}
\end{align*}
where in the third inequality we used Lemma \ref{lem:probaevents}, and in the last inequality we used the fact that $\frac{1-\sigma^2}{\sigma^2}\wedge 1\geq \frac{1-\sigma^2}{1+2\sigma\sqrt{1-\sigma^2}}$.
\end{proof}

\section{Proof of Lemma \ref{lem:not_rc_dom}}\label{app:lem_not_rc_dom}
The proof of Lemma \ref{lem:not_rc_dom} uses elements of the proof of Proposition \ref{prop:diago_dom}. The interested reader is invited to read the proof of Proposition \ref{prop:diago_dom} first. 
\begin{proof}[Proof of Lemma \ref{lem:not_rc_dom}]
It will be useful to first generalize our notation. For that, we denote \[C_{ij,x}=(AXB)_{ij}, \enskip C'_{ij,x}=(AXA)_{ij},\enskip C''_{ij,x}=(AXZ)_{ij}
\] for $x\in \calS_n$, and \[\mathcal{E}^{ij}_{x^{-1}}:=\{\sqrt{1-\sigma^2}C'_{ij,x^{-1}}\geq s_{\sigma,x}/2 \}\cup \{\sigma C''_{ij,x^{-1}}\geq s_{\sigma,x}/2 \}\]
where $x^{-1}$ is the inverse permutation of $x$.
The fact that $\prob(C_{ii,x}<C_{ij,x})\leq8e^{-c(\sigma)s_x^2n} $ follows directly from the bound for $\tilde{\mathcal{E}}_j$ derived in the proof of Proposition \ref{prop:diago_dom} part $(ii)$. To prove $\prob(C_{ii,x}<C_{ji,x})$ notice that $C'_{ji,x}=C'_{ij,x^{-1}}$ and that $C''_{ji,x}\stackrel{d}{=}C''_{ij,x^{-1}}$. On the other hand, notice that $s_x=s_{x^{-1}}$ (hence $s_{\sigma,x}=s_{\sigma,x^{-1}}$). Arguing as in Lemma \ref{lem:probaevents} it is easy to see that 
\[\prob(C_{ii,x}<C_{ji,x})\leq 8e^{-c(\sigma)s_x^2n}.\]
The bound on $\prob(\exists j,\text{ s.t }C_{ij,x}\vee C_{ji,x}>C_{ii,x})$ then follows directly by the union bound. 
\end{proof}

\section{Proofs of Lemmas \ref{lem:diagdom_LAP} and \ref{lem:overlap_event}} \label{app:proofs_lem_diagdom}
\begin{proof}[Proof of Lemma \ref{lem:diagdom_LAP}]
By assumption $C$ is diagonally dominant, which implies that $\exists i_1$ such that $C_{i_1i_1}=\max_{i,j}C_{ij}$ (in other words, if the largest entry of $C$ is in the $i_1$-th row, then it has to be $C_{i_1i_1}$, otherwise it would contradict the diagonal dominance of $C$). In the first step of $\operatorname{GMWM}$ we select $C_{i_1i_1}$, assign $\pi(i_1)=i_1$ and erase the $i_1$-th row and column of $C$. By erasing the $i_1$-th row and column of $C$ we obtain a matrix which is itself diagonally dominant. So by iterating this argument we see $\exists$ $i_1,\cdots,i_n\subset[n]$ such that $\pi(i_k)=i_k$, for all $k$, so $\pi$ has to be the identical permutation. This proves that if $C$ is diagonally dominant, then $\Pi=\operatorname{Id}$. By using the contrareciprocal, \eqref{eq:probneqId} follows.
\end{proof}

\begin{proof}[Proof of Lemma \ref{lem:overlap_event}]

We argue by contradiction. Assume that for some $1\leq k\leq r$, we have $\pi(i_k)\neq i_k$ (and $\pi^{-1}(i_k)\neq i_k$). This means that at some some step $j$ the algorithm selects either $C^{(j)}_{i_k\pi(i_k)}$ or $C^{(j)}_{\pi^{-1}(i_k)\pi(i_k)}$ as the largest entry, but this contradicts the row-column dominance of $i_k$. This proves that that if there exists a set of indices $I_r\subset[n]$ of size $r$ such that for all $i\in I_r$, $C_{ii}$ is row-column dominant, then that set is selected by the algorithm, which implies that $\pi(i)=i$ for $i\in I_r$, thus $\operatorname{overlap}(\pi,\operatorname{id})\geq r$. \eqref{eq:overlap_event} follows by the contrareciprocal. 

\end{proof}

\section{Additional technical lemmas}\label{sec:additionalLemmas}

Here we gather some technical lemmas used throughout the paper. 

\subsection{General concentration inequalities}

The following lemma corresponds to \citep[Lemma 1.1]{LauMass} and controls the tails of the weighted sums of squares of Gaussian random variables.
\begin{lemma}[Laurent-Massart bound]\label{lem:lau_mass}
Let $X_1,\cdots,X_n$ be i.i.d standard Gaussian random variables. Let $\mu=(\mu_1,\cdots,\mu_n)$ be a vector with non-negative entries and define $\zeta=\sum^n_{i=1}\mu_i(X^2_i-1)$. Then it holds for all $t\geq 0$ that
\begin{align*}
    \prob(\zeta\geq 2\|\mu\|_2\sqrt{t}+2\|\mu\|_\infty t)\leq e^{-t}\\
    \prob(\zeta\leq -2\|\mu\|_2\sqrt{t})\leq e^{-t}
\end{align*}
 %
\end{lemma}
An immediate corollary now follows.
\begin{corollary}\label{cor:lau_mass}
Let $X_1,\cdots,X_{n_1}$ and $Y_1,\cdots,Y_{n_2}$ be two independent sets of i.i.d standard Gaussian random variables. Let $\mu=(\mu_1,\cdots,\mu_{n_1})$ and $\nu=(\nu_1,\cdots,\nu_{n_2})$ be two vectors with non-negative entries. Define $\zeta=\sum^{n_1}_{i=1}\mu_iX^2_i$ and $\xi=\sum^{n_2}_{i=1}\nu_iY^2_i$. Then it holds for $t\geq 0$ that
\begin{align}\label{eq:lau_mass_ut}
    \prob\big(\zeta-\xi\geq \|\mu\|_1-\|\nu\|_1+2(\|\mu\|_2+\|\nu\|_2)\sqrt{t}+2\|\mu\|_\infty t\big)&\leq 2e^{-t}, \\ \label{eq:lau_mass_lt}
    \prob\big(\zeta-\xi\leq \|\mu\|_1-\|\nu\|_1-2(\|\mu\|_2+\|\nu\|_2)\sqrt{t}-2\|\nu\|_\infty t\big)&\leq 2e^{-t}.
\end{align}
\end{corollary}

The next lemma give us a distributional equality for terms of the form $\langle g, X g\rangle $ where $g$ is a standard Gaussian vector and $X$ is a permutation matrix. 
\begin{lemma}\label{lem:dist_gaussian_inner}
Let $X\in\calP_n$ and $g=(g_1,\cdots,g_n)$ be a standard Gaussian vector. Then is holds \[\langle g,Xg\rangle\stackrel{d}{=}\sum^{n}_{i=1}\lambda_ig'^2_i,\]
where $\lambda_i$ are the eigenvalues of $\frac12(X+X^T)$ and $g'=(g_1,\cdots,g_n)$ is a vector of independent standard Gaussians. Moreover, if $|S_X|=s_xn$ for $s_x\in(0,1]$,   $\mu\in \mathbb{R}^{n_1}$ is a vector containing the positive eigenvalues of $\frac12(X+X^T)$, and $-\nu\in \mathbb{R}^{n_2}$ is a vector containing the negative eigenvalues of $\frac12(X+X^T)$, then
\begin{align*}
    \|\mu\|_1-\|\nu\|_1&=s_xn,\\
   \sqrt{n}\leq\|\mu\|_2+\|\nu\|_2&\leq \sqrt{2n},\\
    \|\mu\|_\infty,\|\nu\|_\infty&\leq 1 .
\end{align*}
\end{lemma}
\begin{proof}
Notice that $\langle g,Xg\rangle=\langle g,\frac12(X+X^T)g\rangle$ and given the symmetry of the matrix $\frac12(X+X^T)$ all its eigenvalues are real. Take its SVD decomposition $\frac12(X+X^T)=V\Lambda V^T$. We have that \begin{align*}
    \langle g,\frac12(X+X^T)g\rangle&= (V^Tg)^T\Lambda V^Tg
    \stackrel{d}{=}\sum^n_{i=1}\lambda_ig'^2_i
\end{align*}
using the rotation invariance of the standard Gaussian vectors. Notice that \[|S_X|=Tr(X)=Tr\left(\frac12(X+X^T)\right)=\sum^n_{i=1}\lambda_i\]
which leads to \[\|\mu\|_1-\|\nu\|_1=\sum^n_{i=1}\lambda_i=|S_X|=s_xn.\]
The fact that $\|\mu\|_\infty,\|\nu\|_\infty\leq 1$ follows easily since $X$ is a unitary matrix. The inequality $\|\mu\|_2+\|\nu\|_2\geq \sqrt{n}$ follows from the fact that $\|\mu\|_2^2+\|\nu\|_2^2=n$. From the latter, we deduce that $\|\mu\|_2+\|\nu\|_2\leq \sqrt{\|\mu\|^2_2}+\sqrt{n-\|\mu\|^2_2}\leq 2\sqrt{\frac{n}2}$, and the result follows. 
\end{proof}

\subsection{Concentration inequalities used in Theorem \ref{thm:unif_rec_ppm}}\label{sec:app_thm3}
In this section we provide proofs of Lemma's \ref{lem:nb_ngbh1_mt} and \ref{lem:nb_ngbh2_mt} used to prove Theorem \ref{thm:unif_rec_ppm}.

\begin{proof}[Proof of Lemma's \ref{lem:nb_ngbh1_mt} and \ref{lem:nb_ngbh2_mt}.]
Recall that $B_{ij}= \sqrt{1-\sigma^2}A_{ij}+\sigma Z_{ij}$.
\paragraph{Step 1.} First let us consider the terms of the form $\langle A_{i:},A_{i:} \rangle$. We can write 
\[ \langle A_{i:},A_{i:} \rangle = \sum_{i=1}^{n-1}\mu_i g_i^2\]
where $g_i$ are independent standard Gaussian random variables and $\mu_i =1/n$ for all $i$. Observe that $||\mu||_2=\sqrt{\frac{n-1}{n^2}}$.
By Lemma \ref{lem:lau_mass} we have for $i \in [n]$ and all $t>0$
\[ \prob\left(\langle A_{i:},A_{i:} \rangle \leq \frac{n-1}{n}-2\sqrt{\frac{t(n-1)}{n^2}}\right)\leq e^{-t}.\]
For the choice $t= 5\log n$ we obtain \[ \langle A_{i:},A_{i:} \rangle \geq 1-O\left(\sqrt{\frac{\log n}{n}}\right)\] with probability at least $1-e^{-5\log n}$.

\paragraph{Step 2.} Let us consider now terms of the form $\langle A_{i:},Z_{i:} \rangle$. We can write 
\[ \langle A_{i:},Z_{i:} \rangle = \frac{1}{n}\sum_{i=1}^{n-1} (g_ig_i') = \frac{1}{n} G^\top G' \]
where $G=(g_i)_{i=1}^{n-1}$ and $G'=(g'_i)_{i=1}^{n-1}$ are i.i.d. standard Gaussian random variables. We can write \[ G^\top G'= \norm{G}\left( \left(\frac{G}{\norm{G}}\right)^\top G'\right).\] Since $G'$ is invariant by rotation $(\frac{G}{\norm{G}})^\top G'$ is independent from $G$ and has distribution $\calN(0,1)$. By Gaussian concentration inequality we hence have \[ \left(\frac{G}{\norm{G}}\right)^\top G' \leq C\sqrt{\log n}\] with probability at least $1-e^{-5\log n}$ for a suitable choice of $C$. Similarly, by Lemma \ref{lem:lau_mass} we have \[ \norm{G} \leq 2\sqrt{n} \] with probability at least $1-e^{-5\log n}$. Hence with probability at least $1-2e^{-5\log n}$ we have \[ \frac{1}{n} G^\top G' \leq 2C\sqrt{\frac{\log n}{n}}.\]

\paragraph{Step 3.} The same argument can be used to show that for $i\neq j$
 \[ \prob\left(\langle A_{i:},A_{j:} \rangle \geq C\sqrt{\frac{\log n}{n}} \right)\leq e^{-5\log n}.\]

\paragraph{Conclusion.} We can conclude by using the identity $B_{ij}= \sqrt{1-\sigma^2}A_{ij}+\sigma Z_{ij}$ and taking the union bound over all indices $i\neq j$.
\end{proof}

\end{document}